\documentclass[a4,12pt]{amsart}
\usepackage{amssymb,amstext,amsmath,amscd,amsthm,amsfonts,enumerate,graphicx,latexsym, hyperref,cleveref}
\usepackage[usenames]{color}
\usepackage[all]{xy}

 \usepackage{comment}
 \usepackage{mathrsfs}

\newtheorem{thm}{Theorem}[section]

\newtheorem{cor}[thm]{Corollary}

\newtheorem{lem}[thm]{Lemma}
\newtheorem{prop}[thm]{Proposition}
\newtheorem{ques}[thm]{Question}
\newtheorem{theoremx}{Theorem}
\newtheorem*{theorem*}{Theorem}

\theoremstyle{definition}
\newtheorem{dfn}[thm]{Definition}
\newtheorem{rem}[thm]{Remark}

\newtheorem{ex}[thm]{Example}

\newtheorem*{claim*}{Claim}
\theoremstyle{remark}

\newcommand{\etalchar}[1]{$^{#1}$}
\numberwithin{equation}{thm}
\def\cm{\operatorname{CM}}
\def\rf{\operatorname{Ref}}
\def\cond{\mathfrak c}

\def\End{\operatorname{End}}
\def\Ext{\operatorname{Ext}}

\def\Hom{\operatorname{Hom}}

\def\id{\mathrm{id}}
\def\im{\operatorname{Im}}

\def\m{\mathfrak{m}}
\def\Min{\operatorname{Min}}
\def\mod{\operatorname{mod}}

\def\spec{\operatorname{Spec}}
\def\syz{\Omega}

\def\tr{\operatorname{tr}}

\def\tracecat{\operatorname{T}}
\def\reftracecat{\operatorname{RT}}
\def\ulrichcat{\operatorname{Ul}}

\def\core{\operatorname{core}}
\def\br{\operatorname{Bir}}
\def\type{\operatorname{type}}
\def\cmf{\operatorname{CM_{full}}}
\def\supp{\operatorname{Supp}}
\def\chr{\operatorname{char}}

\newcommand{\ses}[3]{0 \to {#1} \to {#2} \to {#3} \to 0}
\newcommand{\ds}{\displaystyle}
\begin{document}
%\allowdisplaybreaks
\setlength{\baselineskip}{14pt}
\title{On reflexive and $I$-Ulrich modules over curve singularities}
%\date{September 4, 2014}
\author{Hailong Dao}
\address{Department of Mathematics, University of Kansas, Lawrence, KS 66045-7523, USA}
\email{hdao@ku.edu}
\urladdr{https://www.math.ku.edu/~hdao/}

\author{Sarasij Maitra}
\address{Department of Mathematics, University of Virginia, 141 Cabell Drive, Kerchof Hall
	P.O. Box 400137
	Charlottesville, VA 22904}
\email{sm3vg@virginia.edu}
\urladdr{https://sarasij93.github.io/}

\author{Prashanth Sridhar}
\address{
Department of Mathematics, University of Kansas, Lawrence, KS 66045-7523, USA}
\address{School of Mathematics, Tata Institute of Fundamental Research, Homi Bhabha Road, Mumbai, India 400005}

\email{prashanth.sridhar0218@gmail.com}
\urladdr{https://sites.google.com/view/prashanthsridhar/home}
%\thanks{2010 {\em Mathematics Subject Classification.} 13C14, 13H10}

\begin{abstract}
We study reflexive modules over one dimensional Cohen-Macaulay rings. Our key technique exploits the concept of $I$-Ulrich modules. 
\end{abstract}
\keywords{reflexive modules, $I$-Ulrich modules, trace ideals, conductors, birational extensions, Gorenstein rings, canonical ideal.} %
\subjclass[2010]{}
\maketitle
%\tableofcontents
%%%%%%%%%%%%%%%%%%%%%%%%%%%%%%%%%%%%%%%%%%%%%%%%%%
%\section*{Convention}

%%%%%%%%%%%%%%%%%%%%%%%%%%%%%%%%%%%%%%%%%%%%%%%%%%%%%%%%%%%%%%%
\section{Introduction}
Notions of reflexivity have been studied throughout various branches  of  mathematics. Over a commutative ring $R$, recall that a module $M$ is called reflexive if the natural map $M \to M^{**}$ is an isomorphism, where $M^*$ denotes  $\Hom_R(M,R)$. When $R$ is a field, any finite dimensional vector space is reflexive, a fundamental fact in linear algebra. Over general rings,  from what we found in the  literature, these modules were studied  in the works of Dieudonn{\'e} \cite{dieudonne1958remarks}, Morita \cite{morita1958duality} and Bass (\cite{bass84}, where the name ``reflexive" first appeared) before being treated formally in \cite{bourbaki1965diviseurs}. They are now  classical and ubiquitous  objects in modern commutative algebra and algebraic geometry. In this paper we seek to understand the following basic question: how common are they?

Note that since being reflexive  is equivalent to being $(S_2)$ and reflexive in codimension one (over an $(S_2)$ ring, see \cite[Proposition 1.4.1]{bruns_herzog_1998}), understanding the one dimensional (local) case is key to understanding reflexivity in general.

Assume now that $(R,\m)$ is a one dimensional Noetherian local Cohen-Macaulay ring. The primary examples are curve singularities or localized coordinate rings of points in projective space. It turns out that the answers to our basic question can be quite subtle. If $R$ is Gorenstein, then any maximal Cohen-Macaulay module is reflexive (so in our dimension one situation, any torsionless module is reflexive). However, reflexive modules or even ideals are poorly understood when $R$ is not Gorenstein. Some general facts are known, for instance if $R$ is reduced, any second syzygy or $R$-dual module is reflexive. However, we found very few concrete examples in the literature: only the maximal ideal and the conductor of the integral closure of $R$. How many can there be? Can we classify them? When is an ideal of small colength reflexive?  For instance, if $R$ is $\mathbb C[[t^3,t^4,t^5]]$ then any indecomposable reflexive module is either  isomorphic to $R$ or the maximal ideal, but the reason is far from clear.

Our results will give answers to the above questions in many cases. A key point in our investigation is  a systematic application of the concept of $I$-Ulrich modules, where $I$ is any  ideal of height one in $R$. A module $M$ is called $I$-Ulrich if $e_I(M)=\ell(M/IM)$ where $e_I(M)$ denotes the Hilbert Samuel multiplicity of $M$ with respect to $I$ and $\ell(\cdot)$ denotes length.  This is a straight generalization of the notion of Ulrich modules, which is just the case $I=\m$ (\cite{ulrich1984gorenstein} and  \cite{brennan1987maximally}). Note also that $I$ is $I$-Ulrich simply says that $I$ is stable, a concept heavily used in Lipman's work on Arf rings \cite{lipman1971stable}. 

Of course, the study of $\m$-Ulrich modules and certain variants has been an active area of research for quite some time now. The papers closest to the spirit of our work are perhaps \cite{goto2014ulrich},\cite{goto2016ulrich},\cite{herzog1991linear}, \cite{kobayashi2019ulrich}, \cite{nakajima2017ulrich} among many other sources.

We will show that $\omega_R$-Ulrich modules are reflexive, and they form a category critical to the abundance of reflexive modules (here $\omega_R$ is a canonical ideal of $R$). For instance, any maximal Cohen-Macaulay module over an $\omega_R$-Ulrich finite extension of $R$ is reflexive. Furthermore, a reflexive birational extension of $R$ is Gorenstein  if and only if its conductor $I$ is $I$-Ulrich and $\omega_R$-Ulrich. 

We also make frequent use of birational extensions of $R$ and trace ideals. This is heavily inspired by some recent interesting work from Kobayashi \cite{kobayashi2017syzygies}, Goto-Isobe-Kumashiro \cite{goto2020correspondence}, Faber \cite{faber2019trace} and Herzog-Hibi-Stamate \cite{herzog2019trace}.

We now describe the structure and the main results of our paper. Sections $2$ and $3$ collect basic results on reflexive modules, trace ideals and birational extensions to be used in later sections. Section $4$ develops the concept of $I$-Ulrich modules for any ideal $I$ of height one in $R$, see \Cref{defIulrich}. We give various characterizations of $I$-Ulrichness (\Cref{I-UlrichThm}).  We show the closedness of the subcategory of $I$-Ulrich modules under various operations, prompting the existence of a lattice like structure for $I$-Ulrich ideals, which can be referred to as an \textit{Ulrich} lattice. We establish tests for $I$-Ulrichness using blow-up algebras and the \textit{core} of $I$ (recall that the core of an ideal is the intersection of all minimal reductions). Finally, we show that an $\omega_R$-Ulrich $M$ satisfies $\Hom_R(M,R) \cong 
\Hom_R(M,\omega_R)$, and that such a module is reflexive. 

The later sections deal with applications. In Section $5$, under mild conditions, we are able to completely characterize  extensions $S$ of $R$ that are ``strongly reflexive" in the following sense: any maximal Cohen-Macaulay $S$-module is reflexive over $R$. Interestingly, in the birational case, this classification involves the core of the canonical ideal of $R$. 

\begin{theoremx}[\Cref{totallyreflexive} and \Cref{refthm}]\label{t1}
Suppose that $R$ is a one-dimensional Cohen-Macaulay local ring with a canonical ideal $\omega_R$. Let $S$ be a module finite $R$-algebra such that $S$ is a maximal Cohen-Macaulay module over $R$. The following are equivalent (for the last two, assume that  $S$ is a birational extension and the residue field of $R$ is infinite):

\begin{enumerate}
    \item Any  maximal Cohen-Macaulay $S$-module is $R$-reflexive. 
    \item $\omega_S$ is $R$-reflexive.
    \item $\Hom_R(S,R)\cong \Hom_R(S,\omega_R)$.
    \item $\omega_RS \cong S$.
    \item $S$ is $\omega_R$-Ulrich as an $R$-module.
    \item $S$ is $R$-reflexive and the conductor of $S$ to $R$ lies inside $(x):\omega_R$ for some principal reduction $x$ of $\omega_R$.
    \item $S$ is $R$-reflexive and the conductor of $S$ to $R$ lies inside $\core(\omega_R):_R \omega_R$. 
\end{enumerate}
\end{theoremx}

The theorem above extends a result of Kobayashi \cite[Theorem 2.14]{kobayashi2017syzygies}.  Also, for $S$ satisfying one of the conditions of \Cref{t1}, any contracted ideal $IS\cap R$ is reflexive (\Cref{integrallyclosedreflexive}). Such a statement generalizes a result by Corso-Huneke-Katz-Vasconcelos that if $R$ is a domain and the integral closure $\overline R$ is finite over $R$, then any integrally closed ideal is reflexive  \cite[Proposition 2.14]{corso2005integral}. 

Our next Section $6$ deals with various ``finiteness results", where we study when certain subcategories or subsets of $\cm(R)$ are finite or finite up to isomorphism. One main result roughly says that if the conductor of $R$ has small colength, then there are only finitely many reflexive ideals that contain a regular element, up to isomorphism. 

\begin{theoremx}[\Cref{small}]\label{t2}
Let $R$ be a one-dimensional Cohen-Macaulay local ring with infinite residue field and  conductor ideal $\cond$. Assume that either:
\begin{enumerate}
\item $\ds \ell(R/\mathfrak{c})\leq 3$, or
\item $\ds \ell(R/\mathfrak{c})=4$ and $R$ has minimal multiplicity.
\end{enumerate}
Then the category of regular reflexive ideals of $R$ is of finite type. 
\end{theoremx}

We also characterize rings with up  to three trace ideals (\Cref{TR3}). We observe that if $S=\End_R(\m)$ has finite representation type, then $R$ has only finitely many indecomposable reflexive modules up to isomorphism (\Cref{propfinRef}). In particular, seminormal singularities have ``finite reflexive type" (\Cref{seminormalcor}).

In Section $7$ we give some further applications on almost Gorenstein rings. We show that in such a ring, all powers of trace ideals are reflexive (\Cref{almostgorprop}). We also characterize reflexive birational extensions of $R$ which are Gorenstein:

\begin{theoremx}[\Cref{Gorclassificationreflexive}]\label{t3}
Suppose that $R$ is a one-dimensional Cohen-Macaulay local ring with a canonical ideal $\omega_R$. Let $S$ be a finite birational extension of $R$ which is reflexive as an $R$-module. Let $I=\cond_R(S)$ be the conductor of $S$ in $R$. The following are equivalent:
\begin{enumerate}
    \item $S$ is Gorenstein. 
    \item $I$ is $I$-Ulrich and $\omega_R$-Ulrich. That is $I\cong I^2\cong I\omega_R$.
\end{enumerate}

\end{theoremx}

We end the paper with a number of examples and open questions.

\section{Preliminaries I: general Noetherian rings}
Throughout this article, assume that all rings are commutative with unity and are Noetherian, and that all modules are finitely generated. 

Let $R$ denote a Noetherian ring with total ring of fractions $Q(R)$. Let $\overline{R}$ denote the integral closure of $R$ in $Q(R)$. Let $\spec{R}$ denote the set of prime ideals of $R$.  For any $R$-module $M$, if the natural map $\ds M \to M \otimes_R Q(R)$ is injective, then $M$ is called \textit{torsion-free}. It is called a torsion module if $\ds M \otimes_R Q(R) = 0$.
The \textit{dual} of $M$, denoted $M^*$, is the module $\Hom_R ( M, R )$; the \textit{bidual} then is $M^{**}$. The bilinear map $$\ds M\times  M^* \to R,\qquad \ds ( x, f )\mapsto f(x),$$ induces
a natural homomorphism $\ds h : M \to M^{**}$. We say that $M$ is \textit{torsionless} if $h$ is injective, and $\ds M$ is \textit{reflexive} if $h$ is bijective. 

For  $R$-submodules $M,N$ of $Q(R)$, we denote 
\begin{align*}
    M:_RN &=\{a\in R~|~aN\subseteq M\}\\
    M:N &=\{a\in Q(R)~|~aN\subseteq M\}.
\end{align*}

We will need the notion of trace ideals. We first recall the definition. 

\begin{dfn}\label{deftrace}
	The \textit{trace ideal} of an $R$-module $M$, denoted $\text{tr}_R(M)$ or simply $\tr(M)$ when the underlying ring is clear, is the image of the map $\tau_M:M^*\otimes_R M\rightarrow R$ defined by $\tau_M(\phi\otimes x)=\phi(x)$ for all $\phi\in M^*$ and $x\in M$.
\end{dfn}	

Say that an ideal $I$ is a trace ideal if $\ds I=\tr(M)$ for some module $M$. Since $\ds \tr(\tr(M))=\tr(M)$, $I$ is a trace ideal if and only if $I=\tr(I)$. It is clear from the definition, that if $M\cong N$, then $\tr(M)=\tr(N)$.

There are various expositions on trace ideals scattered through the literature, see for example \cite{herzog2019trace},\cite{haydee_trace_end},\cite{kobayashi2019rings},\cite{goto2020correspondence},\cite{faber2019trace}, etc. For the purposes of this paper, we shall mainly need the following properties of trace ideals.

\begin{prop}\cite[Proposition 2.4]{kobayashi2019rings}\label{traceproperties}
Let $M$ be an $R$-submodule of $Q(R)$ containing a nonzero divisor of $R$. Then the following statements hold.
\begin{enumerate}
\item $\tr(M)=(R: M)M$.
\item The equality $M = \tr(M)$ holds if and only if  $M : M = R : M$ in $Q(R)$.
\end{enumerate}
\end{prop}

Recall that a finitely generated $R$-submodule $I$ of $Q(R)$ is called a {\it fractional ideal} and it is {\it regular} if it is isomorphic to an $R$-ideal of grade one.
\begin{rem}\label{convention_remark}
	Let $R$ be any ring with total ring of fractions $Q(R)$. For any two regular fractional ideals $I_1,I_2$, we have $\ds I_1:I_2\cong \Hom_R(I_2,I_1)$ where the isomorphism is as $R$-modules, see for example \cite[Lemma 2.1]{herzog238kanonische}. 

	By abuse of notation, we will identify these two $R$-modules and use them interchangeably in the remainder of the paper.
\end{rem}

\begin{rem}\label{convention_remark_2}
	By \Cref{convention_remark}, we can identify $I^{*}:=\Hom_R(I,R)$ with $R:I$ whenever $I$ contains a non zero divisor. This is also denoted as $I^{-1}$. Moreover for any non zero divisor $x$ in $I$, we have $\ds xI^*=x:_R I$. Hence $I^*\cong x:_R I$ as $R$-modules.
\end{rem}

Next, we discuss some general statements about reflexive modules that will be needed. Recall that an $R$-module $M$ is called {\it totally reflexive} if $M$ is reflexive and $\Ext^{i}_R(M,R)=\Ext^i_R(M^*,R)=0$ for all $i>0$ (see \cite{kustin2018totally} for instance for more details). In the result below we need to consider modules that are locally totally reflexive on the minimal primes of $R$. See \Cref{refrem}.

\begin{lem}\label{reflem}
Let $R$ be a Noetherian ring satisfying condition $(S_1)$. Consider modules $M,N$ that are locally totally reflexive on the minimal primes of $R$. 
\begin{enumerate}
\item Assume there is a short exact sequence $\ses{M}{N}{C}$. If $N$ is reflexive and $C$ is torsionless then $M$ is reflexive.
\item If $\Hom_R(M,N)$ is locally totally reflexive on the minimal primes of $R$ and $N$ is reflexive, then $\Hom_R(M,N)$ is reflexive. 
\end{enumerate} 
\end{lem}

\begin{proof}
For $(1)$, \cite[Theorem 29]{masek1998gorenstein} implies that $C$ is locally totally reflexive on the minimal primes. Thus $\Ext^1_{R_P}(C_P,R_P)=0$ for all minimal primes $P$ and the conclusion follows from  \cite[Proposition 8]{masek1998gorenstein}.

For $(2)$, we first prove the case $N=R$. Let $f: M^* \to M^{***}$ be the natural map. Let $g: M^{***}\to M^*$ be the dual of the natural map $M\to M^{**}$. Then $g\circ f=\id$, so $f$ splits, and we get $M^{***}=M^*\oplus M_1$. But for any minimal prime $P$, $(M_1)_P=0$, so $M_1$ has positive grade. As $M^{***}$ embeds in a free module and $R$ is $(S_1)$, $M_1=0$.

Now, start with a short exact sequence $\ses{C}{F}{N^*}$ where $F$ is free. Dualizing, we get $\ses{N}{F^*}{D}$ where $D$ is torsionless (by \cite[Proposition 8]{masek1998gorenstein},  a submodule of a torsionless module is torsionless). Take $\Hom_R(M,-)$ to get the exact sequence $\ses{\Hom_R(M,N)}{\Hom_R(M,F^*)}{K'}$ where $K'$ is a sub-module of $\Hom_R(M,D)$. We can apply part $(1)$ and the previous paragraph to get that $\Hom_R(M,N)$ is reflexive provided we show that $K'$ is torsionless.

Since, $D$ is torsionless, $D$ embeds into $D^{**}$ and hence into a free module say $G$. Thus, $\Hom_R(M,D)$ embeds into $\Hom_R(M,G)$. Finally, note that $M^*$ is torsionless as it is a submodule of a free module and thus, by applying  \cite[Proposition 8]{masek1998gorenstein} twice, we get that $K'$ is torsionless. 
\end{proof}

\begin{rem}\label{refrem}
 We stated \Cref{reflem} in quite general setting. One reason is in practice, as well as in this paper, it is often applied in the following two different situations: when $R$ is generically Gorenstein, or when $M,N$ are locally free on the minimal primes (for instance when they are birational extensions or regular ideals of $R$). In such situations the totally reflexive assumptions are automatically satisfied. 
\end{rem}

The relationship between reflexive modules and birational extensions is also naturally of interest.
We  say that an extension $f: R \to S$ is {\it birational} if $S\subset Q(R)$. Equivalently $Q(R)=Q(S)$. 

\begin{prop}\cite[Proposition 2.4]{kobayashi2019rings}\label{fractionalreflexive}
Let $M\subseteq Q(R)$ be a regular fractional ideal of $R$. Then $M$ is a reflexive $R$-module if and only if there is an equality $M=R:(R:M)$ in $Q(R)$.
\end{prop}

The following lemma was stated in more generality in a recent work of S. Goto, R. Isobe, and S. Kumashiro. 
\begin{lem}\label{goto2020}\cite[Lemma 2.6(1), Proposition 2.9]{goto2020correspondence}.
Let $R$ satisfy $(S_1)$ and $f:R\to S$ be a finite birational extension. Then the conductor of $S$ to $R$, denoted $\ds \mathfrak{c}_R(S):=R: S$, is a reflexive regular trace ideal of $R$. Thus, we get a bijective correspondence between reflexive regular trace ideals of $R$ and reflexive birational extensions of $R$ via the map $\alpha: I\mapsto \ds \End_R(I)$ and its inverse $\beta: S \mapsto \cond_R(S)$.
\end{lem}

\begin{proof}
Let $S$ be a reflexive birational extension. Then $ \cond_R(S)\cong S^*$ is reflexive by  \Cref{reflem}. Next, note that $\ds \tr(S)=S^*S=\mathfrak{c}_R(S)S=\mathfrak{c}_R(S)$. So $\cond_R(S)$ is a reflexive trace ideal. If $I$ is a regular reflexive trace ideal, then \Cref{traceproperties}(2) and \Cref{reflem}(2) tell us that $\ds \End_R(I)$ is indeed reflexive.

Finally, we have $\beta(\alpha(I)) = I$ and $\alpha(\beta(S)) = S$ by \Cref{traceproperties},   \Cref{fractionalreflexive} and the above paragraph. 
\end{proof}

So these birational extensions provide important sources for generating reflexive ideals. We have the following criteria for a reflexive module over $R$ to be a module over a finite birational extension $S$. This was stated in \cite[Theorem 3.5]{faber2019trace} for reduced one dimensional local rings, but the result holds for more general rings, and we restate it here with a self-contained proof: 
\begin{thm}\label{faber}
Let $R$ be a Noetherian ring and $M$ be a finite $R$-module. Let $S$ be a finite birational extension of $R$. Consider the following statements.
\begin{enumerate}
    \item $M$ is a module over $S$.
    \item $\tr(M)\subseteq \cond_R(S)$ where $\cond_R(S)=R:S$.
\end{enumerate}
Then $(1)$ implies $(2)$. The converse is true if $M$ is a reflexive $R$-module. 

\end{thm}
\begin{proof}
Let $M$ be a module over $S$. Then there exists a $S$-linear (hence $R$-linear) surjection $S^n\rightarrow M$, so $\tr(M)\subseteq \tr(S)=\cond_R(S)$.

Conversely assume $\tr(M)\subseteq \cond_R(S)$ and $M$ reflexive. Consider $f\in M^*$ and $s\in S$, we have $s\cdot f\in M^*$ by assumption. Therefore $M^*$ is an $S$-module. From the forward implication, $\tr(M^*)\subseteq \cond_R(S)$. So repeating the argument again, we get that $M^{**}$ is an $S$-module. Since $M$ is reflexive, we are done.  
\end{proof}

\section{Preliminary II: dimension one}
Throughout the rest of this paper (unless otherwise specified), $(R,\m,k)$ will denote a Cohen-Macaulay local ring of dimension one with maximal ideal $\m$ and residue field $k$. Denote by $Q(R)$ the total quotient ring of $R$. For an $\m$-primary ideal $I$ and a module $M$, let $e_I(M)$ denote the Hilbert-Samuel multiplicity of $M$ with respect to $I$. In the case, when $I=\m$, we write $e(M)$. Let $\mathfrak{c}:=R: \overline{R}$ denote the conductor ideal of  $\overline{R}$ to $R$. For an $R$-module $M$, let $\ds \mu(M)$ and  $\ell(M)$,  denote the minimal number of generators of $M$ and the length of $M$ respectively, as an $R$-module. 

 Let $\cm(R)$ denote the category of maximal Cohen-Macaulay $R$-modules and let $\rf(R)$ denote the category of reflexive $R$-modules. 
%  Let $\rf_r(R)$ be the category of rank $r$ reflexive modules. 
 We say a category is of \textit{finite type} if it has only finitely many indecomposable objects up to isomorphism. 
 
 \begin{rem}\label{closetorem}
 Note that the following statements are true.
 
 \begin{enumerate}
     \item $\ds \{\text{free $R$- modules}\}\subset \rf(R)\subset \cm(R).$
     
     \item $R$ is regular if and only if $\ds \{\text{free $R$- modules}\}=\cm(R)$.
     
     \item $R$ is Gorenstein if and only if $\ds \rf(R)=\cm(R)$. 
 \end{enumerate}  
 
 We shall be interested in the behaviour of $\rf(R)$ in the case when $R$ is ``close to" being regular or Gorenstein. We will come back to this in \Cref{firstapps} and \Cref{furtherapps}.
 \end{rem}

The conductor and maximal ideals are natural examples of reflexive trace ideals. 
\begin{cor}\label{c&m}
Let $(R,\m)$ be a Noetherian local ring of arbitrary dimension satisfying $(S_1)$. If $\bar{R}$ is finite over $R$, $\cond$ is a regular reflexive trace ideal. If $\dim R=1$ and $R$ is not regular, $\m$ is a regular reflexive trace ideal.
\end{cor}
\begin{proof}
The first statement follows from \Cref{goto2020}. For the second statement, since $\text{grade}(\m)=1$, $R\subsetneq \m^*$ and hence $\m$ is reflexive by \Cref{fractionalreflexive}. Since $\m\subseteq \tr(\m)$ and $\tr(\m)=R$ if and only if $\m$ is principal, $\m$ is a trace ideal.
\end{proof}

\subsection*{Support and trace}

Let $\cmf(R) = \{ M\in \cm(R) \ | \ \supp(M)= \spec{R}\}$ denote the subcategory of $\cm(R)$ of modules with full support.  

\begin{lem}\label{tracereg}
Suppose that  $M\in \cm(R)$. Then $\tr(M)$ is a regular ideal if and only if $M_P$ has an $R_P$-free summand for each $P\in \Min(R)$.  Thus if $R$ is reduced then  $M\in \cmf(R)$ if and only if $\tr(M)$ is a regular ideal. 
\end{lem}

\begin{proof}
As we are in dimension one, clearly $\tr(M)$ is regular if and only if $\tr(M)_P=R_P$ for any $P\in \Min(R)$. As trace localizes, we have $\tr(M)_P = \tr_{R_P}(M_P)$, and the result follows. 
\end{proof}

\begin{rem}\label{remfull}
Here we discuss why when studying $\cm(R)$, one can reduce to the case of $\cmf(R)$ and hence regular trace ideals thanks to \Cref{tracereg}. Let $\Min(R)=\{P_1,\dots, P_n\}$ denote the set of minimal primes of $R$ and $(0)= \cap Q_i$ with $\sqrt{Q_i}=P_i$. For a subset $X \subset \Min(R)$, let $R_X = R/\cap_{P_i\in X}{Q_i}$. Then $\cm(R) = \cup_{X\subset \Min(R)} \cmf(R_X)$. Thus, understanding $\cm(R)$ amounts to understanding $\cmf(R_X)$ for all subsets $X$.  
\end{rem}

It is well known that $\mathfrak{c}$ and $\m$ are reflexive trace ideals (\Cref{c&m}). In particular, we can investigate other such ideals. 
We set up some further notation which we will use throughout. 
\begin{align*}
 \tracecat(R) &:=\{I~|~ I~\text{is a regular trace ideal}\}\\
 \reftracecat(R) &:= \{I~|~ I~\text{is a regular reflexive trace ideal}\}
 \end{align*}

Note that if $R$ is a complete local domain, then from \cite[Theorem 2.5]{maitra2020partial} we get that for any ideal $I\subset R$, $\ds I^{**}$ is isomorphic to an ideal which contains the conductor $\mathfrak{c}$. This suggests an immediate link, relevant to our study, with the conductor ideal $\mathfrak{c}$.  The following theorem gives a generalization to this fact.

\begin{thm}\label{thMai}
Let $R$ be a one-dimensional Cohen-Macaulay local ring with conductor ideal $\cond$. Any regular ideal $I$ that contains a principal reduction is isomorphic to another fractional ideal $J$ such that $R:J$ (which is isomorphic to $I^*$) contains the conductor $\cond$. In particular, if the residue field is infinite, any reflexive regular ideal of $R$ is isomorphic to an ideal containing $\cond$.

\end{thm}

\begin{proof}
Assume $I$ is a regular ideal of $R$ with a principal reduction $x$. Let $J=\frac{I}{x}= \{\frac{a}{x}, a\in I\}$. Clearly $I\cong J$. Since $I^{n+1}=xI^n$ for some $n$, we have $J\subset I^n:I^n\subset \overline R$. But then $R:J \supset R:\overline R=\cond$. The last statement follows by replacing $I$ with $I^*$ and using the fact that $I^{**}\cong I$.
\end{proof}

%In a similar vein, we also have the following proposition.
%\begin{prop}
%Let $S$ be a finite birational extension of $R$ and $\cond_R(S):=R:S$. If for a regular ideal $I$, we have $I^*S\simeq S$, then $I^{**}\cong J$ for some ideal $J$ of $R$ such that $\cond_R(S)\subseteq J$. 
%\end{prop}
%\begin{proof}
%Since $I^*S\simeq S$, we have for any regular element $x\in I$, that $((x):_R I)S$ is a regular principal ideal of $S$. Choose $a\in ((x):_R I)S$ regular such that $a^{-1}((x):_R I)S\subseteq S$. Set $J:=(a^{-1}((x):_R I))^*$, so that $\cond_R(S)\subseteq J$. Since $I^{*}\cong a^{-1}((x):_R I)$ as $R$-modules, we have $I^{**}\simeq J$.
%\end{proof}

%The following extends \cite[Proposition 6.5]{herzog2019trace}.

\begin{cor}\label{tracecatrem}
Let $R$ be as in \Cref{thMai}. For any regular  ideal $I$ with a principal reduction, $\tr(I)\supset \cond$. 
\end{cor}

\begin{proof}
Let $J$ be the fractional ideal as in \Cref{thMai}. Note that $R\subset J$. Since $R:J\supset \cond$, we have $\tr(I)= \tr(J)=J(R:J)\supset R:J\supset \cond$.
\end{proof}

\begin{lem}\label{traceformula}
Let $R$ be as in \Cref{thMai}. Suppose that $I$ is a regular ideal and $x\in I$ be a non zero divisor. Then $\tr(I)= I((x):_R I):_R x$. In particular, $\tr(I)\supseteq (x):_R I$.
\end{lem}

\begin{proof}
Let $J=\frac{I}{x}\cong I$. Then $\tr(I)= \tr(J) = J(R:J) = \frac{I}{x}((x):_R I)$. Thus, $\tr(I)= I((x):_R I):_R x$.  In particular, $\tr(I)\supset x((x):_R I):_R x=(x):_R I$.
\end{proof}

We can classify trace ideals with reduction number one :

\begin{cor}\label{xI}
Let $R$ be as in \Cref{thMai} and $I$ a regular ideal such that $I^2=xI$ for some regular element $x\in I$. Then $\tr I= (x):_RI$. In particular, $I$ is a trace ideal if and only if $(x):_R I=I$. In that case $I\cong I^*$ and hence $I$ is reflexive. 
\end{cor}

\begin{proof}
From \Cref{traceformula}, $(x):_RI\subseteq \tr I$.
%Since $I$ is regular, $x$ must be a regular element. Let $J=\frac{I}{x} \cong I$. Then $= \tr J = JJ^{-1} \supset J^{-1} = R: \frac{I}{x}= $.
On the other hand, $I\tr I = I(II^{-1})=I^2I^{-1}=xII^{-1}=x\tr I$, so $\tr I\subset (x):_RI$. Thus $\tr I= (x):_RI$.
\par The next assertion is now obvious. For the last assertion, note that $(x):_R I\cong I^*$.  
\end{proof}

It is known that under mild assumptions, all integrally closed ideals are reflexive (\cite[Proposition 2.14]{corso2005integral}). The following proposition vastly generalizes this fact  and also provides a way of generating reflexive or trace ideals by contracting ideals from certain birational extensions. For how to find such extensions see \Cref{sre}.

\begin{prop}\label{integrallyclosedreflexive}
Let $R$ be as in \Cref{thMai}. Let $S$ be a finite birational extension such that $\cm(S)\subset \rf(R)$. Let  $I$ be a regular ideal of $R$. Then $IS\cap R \in \rf(R)$. If $I$  contains $\cond_R(S)$, then $IS\cap R \in \reftracecat(R)$. 
\end{prop} 
\begin{proof}
Let $J= IS\cap R$. As $IS \in \cm(S)$, we have $J^{**}\subset (IS)^{**} =IS$, so $J^{**}\subset IS\cap R=J$, hence $J$ is reflexive. If $I$ contains $\cond_R(S)$ then $\cond_R(S) \subset J$. Now, $\tr J= JJ^{-1}\subset J\cond_R(S)^{-1}=JS$, so $\tr J\subset JS\cap R=J$. 
\end{proof}

The next two results are useful for studying colength two ideals. 

\begin{prop}\label{colength2minmult}
Let $(R,\m)$ be as in \Cref{thMai}. Further assume that $R$  has minimal multiplicity with infinite residue field. Let $I$ be a regular ideal of colength two. Then $I$ is reflexive if and only if it is either integrally closed or principal. 
\end{prop}

\begin{proof}
If $I$ is integrally closed then it is reflexive by \Cref{integrallyclosedreflexive}. Now assume $I$ is neither integrally closed nor principal. We can also assume $R$ is not regular, else the statement holds trivially as any non-zero ideal is principal. Then as $I$ is not integrally closed of colength two,  $\overline I=\m$. We can then pick a regular principal reduction $x$ for $\m$ and $I$. Since $R$ has minimal multiplicity and not regular, we note that $\m = \tr(\m)= (x):_R\m$ by \Cref{xI}. On the other hand $(x):_RI\supset (x):_R\m = \m$, so equality occurs. Using \Cref{convention_remark_2}, we get $\ds xI^*=x\m^*$ and hence $I^*=\m^*$ and $I^{**}=\m$. Thus $I$ is not reflexive.
\end{proof}

We classify colength two ideals that are contracted from $\End_R(\m)$.

\begin{prop}\label{contract}
Let $(R,\m)$ be as in \Cref{thMai}. Let $S=\End_R(\m)$ and $I$ be an ideal of colength two. Then $IS\cap R=I$ if and only if $\ell(S/IS)>\type(R)+1$. 
\end{prop}

\begin{proof}
It is clear that $IS\cap R=I$ if and only if $S/IS$ is a faithful $R/I$ module. As $R/I\cong k[t]/t^2$, $S/IS$ decomposes into a direct sum of $k$ and $R/I$, so it is faithful if and only if it is not a direct sum of $k$'s, in other words the length of $S/IS$ is strictly larger than it's number of generators. But $\mu(S) = \ell(S/\m S)= \ell(S/R)+1=\type R+1$.
\end{proof}

\section{$I$-Ulrich Modules}\label{secIulrich}

Throughout this section, we assume that $(R,\m,k)$ is a one-dimensional Cohen-Macaulay local ring. Let $I$ be an ideal of finite colength and $x$ a principal reduction. This section grew out of the realization that the equality $xM=IM$ for certain modules $M$ appears in many situations related to our investigation. For instance, it will turn out that when $I$ is a canonical ideal, such a module is reflexive and the finite extensions satisfying such conditions are ``strongly reflexive", see \Cref{dfnstrong}.    

We shall call these modules $I$-Ulrich, and define them slightly more generally without using principal reductions. Obviously the name and definition are inspired by the very well-studied notion of Ulrich modules, which are $\m$-Ulrich in our sense. Note that our definition is very much a straight generalization of an Ulrich module, and  not as restrictive as those studied in \cite{goto2014ulrich} and \cite{goto2016ulrich}.

\begin{dfn}\label{defIulrich}
    We say that $M\in \cm(R)$ is $I$-Ulrich if $e_I(M)=\ell(M/IM)$. Let $\ulrichcat_I(R)$ denote the category of $I$-Ulrich modules. 

\end{dfn}
Note that if $M\cong N$ in $\cm(R)$, then the same isomorphism takes $IM$ to $IN$, so $\ell(M/IM)=\ell(N/IN)$ for any ideal $I$ and so Ulrich condition is preserved under isomorphism. 

\begin{ex}\label{ulrichex}
Let $M\in \cm{(R)}$. As $\ell(I^nM/I^{n+1}M)=e_I(M)$ for $n\gg0$, it follows that $I^nM$ is $I$-Ulrich for $n\gg 0$.       
      
\end{ex}

\begin{dfn}
Let $B(I)$ denote the blow-up of $I$, namely the ring $\ds \bigcup_{n\geq 0} (I^n:I^n)$. Let $b(I) = \cond_R(B(I))$, the conductor of $B(I)$ to $R$.

\end{dfn}

\begin{rem}\label{B(I)}
  If $x$ is a principal reduction of $I$, then it is well-known that $B(I)= R[\frac{I}{x}]$, \cite[Theorem 1]{barucci1995biggest}. 

\end{rem}

We shall use some standard properties of Hilbert Samuel multiplicity in the proof of the next proposition. The reader can refer to various resources like, \cite{serre1965algebre}, \cite{serre1997algebre}, \cite[\text{mainly} Corollary 4.7.11]{bruns_herzog_1998}, \cite[\text{mainly} Proposition 11.1.0, 11.2.1]{Huneke-Swanson} for further details on multiplicity.

\begin{prop}\label{I-ulrich} Let $R$ be a one-dimensional Cohen-Macaulay local ring. Suppose that $x\in I$ is a (principal) reduction and $M \in \cm(R)$. The following are equivalent:
\begin{enumerate}
\item $M$ is $I$-Ulrich.
\item $IM= xM$. 
\item $IM\subseteq xM$.
\item $IM \cong M$.
\item $M \in \cm(B(I))$ (see \Cref{B(I)action}).
\item $M$ is $I^n$-Ulrich for all $n\geq 1$.
\item $M$ is $I^n$-Ulrich for infinitely many $n$.
\item $M$ is $I^n$-Ulrich for some $n\geq 1$.

\end{enumerate}
\end{prop}

\begin{proof}
As $x$ is a reduction of $I$, $\ell(M/xM)=e_I(M)$. So $(1)$ is equivalent to $\ell(M/xM)=\ell(M/IM)$, or $IM=xM$. The equivalence of $(2)$ and $(3)$ is obvious. Clearly $(2)$ implies $(4)$. Assuming $(4)$, then $M\cong I^nM$ for $n\gg 0$, so $M$ is $I$-Ulrich by \Cref{ulrichex}. We have established the equivalence of $(1)$ through $(4)$. 

Next, $(3)$ is equivalent to $\frac{I}{x}M\subseteq M$. In other words $(3)$ implies that $M\in \cm(B(I))$, since $R[\frac{I}{x}]=B(I)$.
Since $B(I) = B(I^n)$ for any $n\geq 1$, we have $(5) \Rightarrow (6)$. Clearly $(6) \Rightarrow (7) \Rightarrow (8)$. Finally, assume $(8)$. We have $\ell(M/I^nM)=e_{I^n}(M) = ne_I(M)$. Note that for each $i$, $I^iM$ is in $\cm(R)$ and hence, using properties of multiplicities, we get $\ell(I^iM/I^{i+1}M)\leq \ell(I^iM/xI^iM)=e_I(I^iM)=e_I(M)$ for each $i$. Thus, equality must occur for each $i$; in particular, it occurs for $i=0$, which shows that $M$ is $I$-Ulrich. 
\end{proof}

Without any assumption on the existence of a principal reduction, the following still holds:

\begin{thm}\label{I-UlrichThm}
Let $R$ be a one-dimensional Cohen-Macaulay local ring. Let $I$ be a regular ideal and $M \in \cm(R)$. The following are equivalent:
\begin{enumerate}
\item $IM \cong M$.
\item $M$ is $I$-Ulrich.
\item $M$ is $I^n$-Ulrich for all $n\geq 1$.
\item $M$ is $I^n$-Ulrich for infinitely many $n$.
\item $M$ is $I^n$-Ulrich for some $n\geq 1$.
\item $M \in \cm(B(I))$ (see \Cref{B(I)action}).
\end{enumerate}

\end{thm}

\begin{proof}
Assume $(1)$, then $M\cong I^nM$ for $n\gg 0$, so $M$ is $I$-Ulrich by \Cref{ulrichex}. If $(2)$ holds, then we may pass to a local faithfully flat extension of $R$ possessing an infinite residue field and apply \Cref{I-ulrich}, followed by \cite[Proposition 2.5.8]{ega1964ements} to see that $(1)$ holds. The statements $(2)$ to $(5)$ are unaffected by local faithfully flat extensions, so we can enlarge the residue field and apply \Cref{I-ulrich}.  As $I^n$ contains a principal reduction for $n\gg0$ and $B(I)=B(I^n)$ for all $n\geq 1$, \Cref{I-ulrich} implies that $(4)$ and $(6)$ are equivalent.
\end{proof}

\begin{cor}
Let $R$ be as in \Cref{I-UlrichThm}. Let $I$ be a regular ideal. Then $R$ is $I$-Ulrich if and only if $I$ is principal.
\end{cor}

\begin{proof}
 \Cref{I-UlrichThm} implies that $IR\cong R$, so $I$ is principal.
\end{proof}

\begin{rem}\label{B(I)action}
Note that if $M\in \cm(R)$ is $I$-Ulrich, the proofs of \Cref{I-ulrich} and \Cref{I-UlrichThm} show that the action of $B(I)$ on $M$ extends the action of $R$ on $M$. In other words, there is an action of $B(I)$ on $M$ which when restricted to $R$ yields the original action of $R$ on $M$. In particular, if $M\subseteq Q(R)$, multiplication in $Q(R)$ gives an action of $B(I)$ on $M$. 
\end{rem}

We  say that an extension $f: R \to S$ is {\it birational} if $S\subset Q(R)$. Equivalently $Q(R)=Q(S)$. Also such an $f$ induces a bijection on the sets of minimal primes of $S$ and $R$ and $f_P$ is an isomorphism at all minimal primes $P$ of $R$. Let $\br(R)$ denote the set of finite birational extensions of $R$.
\begin{cor}\label{I-ulrichextension}
Let $R$ be as in \Cref{I-UlrichThm}. Let $R\subseteq S$ be a finite  birational extension of rings. Then $S$ is $I$-Ulrich if and only if $B(I)\subseteq S$.
\end{cor}
\begin{proof}
Follows immediately from \Cref{I-UlrichThm} and \Cref{B(I)action}.
\end{proof}

\begin{cor}\label{condref}
Let $R$ be as in \Cref{I-UlrichThm}. Let $I$ be a regular ideal. If $\bar R$ is a finitely generated $R$-module then $\bar R$ and the conductor $\cond$ are $I$-Ulrich.
\end{cor}

\begin{proof}
As $B(I)\subseteq \bar R$, $\bar R \in \cm(B(I))$ and so by \Cref{I-UlrichThm}, $\bar R$ is $I$-Ulrich. Since $\cond\in \cm(\bar R)\subseteq \cm(B(I))$, the conclusion follows. 
\end{proof}

\begin{cor}\label{refIUl1}
Let $R$ be as in \Cref{I-UlrichThm}. If $M \in \cm(R)$ is $I$-Ulrich, then $\tr M\subseteq b(I)$. If $M\in \rf(R)$, then the converse holds. 
\end{cor}
\begin{proof}
This follows from \Cref{I-UlrichThm} and \Cref{faber}.
\end{proof}

\begin{lem}\label{Uexact}
Let $R$ be as in \Cref{I-UlrichThm}. Let $\ses{A}{B}{C}$ be an exact sequence in $\cm(R)$. If $B$ is $I$-Ulrich then so are $A,C$.
\end{lem}

\begin{proof}
We may enlarge the residue field if necessary and assume that $I$ has a principal reduction $x$. Then $x$ is a regular element and hence induces an exact sequence \[\ses{A/xA}{B/xB}{C/xC}.\] $B$ is $I$-Ulrich if and only if $I$ kills the middle module, but if that's the case then $I$ kills the other two as well. 
\end{proof}

\begin{cor}\label{ulrichunderimage}
Let $R$ be as in \Cref{I-UlrichThm}. Let $\ds M\in \ulrichcat_I(R)$. For any nonzero $\ds f\in M^*, \im(f)\in \ulrichcat_I(R)$.
\end{cor}

\begin{cor}\label{ulrichlattice}
Let $R$ be as in \Cref{I-UlrichThm}. Suppose $J,L\in \ulrichcat_I(R)$ are $R$-submodules of an $R$-module $M$. Then $\ds J+L, J\cap L\in \ulrichcat_I(R)$. 
\end{cor}

\begin{proof}
The assertion follows from the short exact sequence $\ses{J\cap L}{J\oplus L}{J+L}$.
\end{proof}{}

\begin{lem}\label{ulrichunderhom}
Let $R$ be as in \Cref{I-UlrichThm}. If $\ds M\in \ulrichcat_I(R)$, then $\ds \Hom_R(M,N)\in \ulrichcat_I(R)$ for any module $N\in \cm(R)$. 
\end{lem}{}

\begin{proof}
As above, we can assume there is  a principal reduction $x$ of $I$. Note that there is an embedding $$\ds \Hom_R(M,N)\otimes_R R/xR \to \ds \Hom_R(M/xM,N/xN)$$ and the latter is killed by $I$ since $\ds M\in \ulrichcat_I(R)$. This shows that $\ds \Hom_R(M,N)\otimes_R R/xR$ is killed by $I$ and this finishes the proof. 
\end{proof}{}

\begin{prop}\label{traceisulrich}
Let $R$ be as in \Cref{I-UlrichThm}. If $\ds M\in \ulrichcat_I{R}$, then $\tr(M)\in \ulrichcat_I(R)$.
\end{prop}{}

\begin{proof}
Since, $\tr(M)$ is the sum of all images of elements in $\ds M^*$, the proof follows immediately from \Cref{ulrichunderimage} and \Cref{ulrichlattice}.
\end{proof}{}

\begin{cor}\label{blattice}
Let $R$ be as in \Cref{I-UlrichThm}. The set of $I$-Ulrich ideals is a lattice under addition and intersection. The largest element is $b(I)$.     
\end{cor}

\begin{proof}
That this set forms a lattice follows from \Cref{ulrichlattice}. For the last assertion,  first note that $b(I)$ is a module over $B(I)$, and then apply \Cref{I-UlrichThm}, \Cref{traceisulrich} and \Cref{refIUl1}.
\end{proof}

\begin{rem}
Aberbach and Huneke \cite{aberbach1996theorem} defined the \textit{coefficient ideal} of $I$ relative to a principal reduction $x$ as the largest ideal $J$ such that $xJ=IJ$. It follows that the coefficient ideal is just $b(I)$.
\end{rem}

From now on we assume that $I$ contains a principal reduction. Recall that the core of $I$, denoted $\core(I)$, is defined as  the intersection of all (minimal) reductions of $I$.

\begin{prop}\label{ulrichinterval}
Let $R$ be as in \Cref{I-UlrichThm}. Assume that $I$ has a principal reduction $x$. Consider  $M\in \ds \ulrichcat_I(R)$. We have, $$ \tr(M) \subseteq (x):_RI \subseteq \tr(I).$$
If the residue field of $R$ is infinite, then:
$$ \tr(M)\subseteq \core(I):_RI \subseteq (x):_RI \subseteq \tr(I).$$
 \end{prop}{}

\begin{proof}
Let  $\ds J=\tr(M)$.  Note that by  \Cref{traceisulrich}, $J \in \ulrichcat_I(R)$ and so $IJ=xJ\subset (x)$ for any principal reduction $x$. So $J\subset \cap ((x):_RI) = \core(I):_R I$. The last inclusion comes from \Cref{traceformula}. If the residue field of $R$ is infinite, then the core is the intersection of all principal reductions, and the second part follows. 
\end{proof}{}

\begin{cor}\label{corinterval}
Let $R$ be as in \Cref{I-UlrichThm}. Suppose that $I$ is a regular ideal with  a principal reduction $x$. Then
 $$b(I)= \tr(b(I)) \subseteq (x):_RI \subseteq \tr(I).$$
 If the residue field of $R$ is infinite, then:
$$b(I)= \tr(b(I))\subseteq \core(I):_RI \subseteq (x):_RI \subseteq \tr(I).$$

\end{cor}

\begin{proof}
Since $b(I) =\tr(B(I)) \in \ulrichcat_I(R)$, the conclusion follows from \Cref{ulrichinterval}.  
\end{proof}

\begin{cor}\label{refIUl2}
Let $R$ be as in \Cref{I-UlrichThm}. Assume that the residue field of $R$ is infinite. Let $M\in \rf(R)$. The following are equivalent.

\begin{enumerate}
    \item $M$ is $I$-Ulrich.
    \item $\tr(M)\subseteq b(I)$.
    \item $\tr(M)\subseteq (x):_RI$ for some principal reduction $x$ of $I$.
    \item $\tr(M)\subseteq (x):_RI$ for any principal reduction $x$ of $I$.
    \item  $\tr(M)\subseteq \core(I):_RI$.
\end{enumerate}
\end{cor}

\begin{proof}
Combining \Cref{refIUl1} and \Cref{ulrichinterval}, we see that the proof would be complete if we show (3) implies (1). If (3) holds, then for any $f\in M^*$, we have that $I\cdot f\subseteq \Hom_R(M,xR)=xM^*$. Therefore by \Cref{I-ulrich}, $M^*$ is $I$-Ulrich. Since $M\in \rf{R}$, by \Cref{ulrichunderhom} $M$ is $I$-Ulrich.
\end{proof}

In light of the above results, it is natural to ask when $b(I)=\tr(I)$. Note that if this is the case, then $(x):_R I$ is independent of the principal reduction. 

\begin{prop}
Let $R$ be as in \Cref{I-UlrichThm}. If $I$ is a regular reflexive trace ideal such that $xI=I^2$ for some $x\in I$, then $b(I)=\tr(I)$. 
\end{prop}

\begin{proof}
As $I^n\cong I$ for all $n>0$. we get that $B(I) = I: I$. By \Cref{goto2020}, $b(I)=I=\tr(I)$. 
\end{proof}

We can moreover relate $b(I)$ with $\core{I}:_RI$ for a large number of cases. 

\begin{thm}
    Let $R$ be a reduced one dimensional ring with infinite residue field $k$. Let $I$ be a regular ideal with reduction number $r$. Assume that $\chr(k)=0$ or $\chr(k)>r$. Then $$b(I)=\core(I):_R I.$$
    \end{thm}
    \begin{proof} By \cite[Theorem 3.4 b]{polini2005formula} we have that, $$\core(I)=x^{n+1}:_R I^n$$ for suitably large $n$, where $x$ is a principal reduction of $I$. Thus, \[\ds \core(I):_R I=x^{n+1}:_R I^{n+1}\cong (I^{n+1})^*\] Now for large $n$, $(I^{n+1})^*$ is $I$-Ulrich by \Cref{ulrichunderhom} and hence $b(I)=\core(I):_R I$ by \Cref{blattice}.
    \end{proof}
    
    The next proposition will help in establishing some finiteness results.

\begin{prop}
Let $R$ be as in \Cref{I-UlrichThm}. Let $I,J$ be regular ideals. Consider the following statements.
\begin{enumerate}
    \item $\ulrichcat_I(R)=\ulrichcat_J(R)$.
    \item $B(I)=B(J)$.
    \item $b(I)=b(J)$.
\end{enumerate}
Then $(1) \iff (2) \implies (3)$. If $R$ is Gorenstein, then all three are equivalent. 
\end{prop}

\begin{proof}
 Recall that $\ulrichcat_I(R)= \cm(B(I))$  by \Cref{I-UlrichThm}. So if $B(I)=B(J)$ then $\ulrichcat_I(R)=\ulrichcat_J(R)$. Now assume $(1)$. Let $S=B(I)$ and $T=B(J)$. Then $S\in \ulrichcat_J(R)= \cm(T)$. Thus $T\subset TS\subset S$. By symmetry $S\subset T$, so $S=T$.
 If $S=T$, then $\cond_R(S)=\cond_R(T)$, so $(3)$ follows from $(2)$.  If $R$ is Gorenstein, then $B(I), B(J)$ are reflexive, so $b(I)=b(J)$ implies $B(I)=B(J)$ by \Cref{goto2020}.
\end{proof}

\begin{thm}\label{ulrichIs}
Let $R$ be as in \Cref{I-UlrichThm}. Let $\cond$ be the conductor and $I$ be a regular ideal. If $\cond \cong I^s$ for some $s$ then $\ulrichcat_I(R)=\cm(\overline R)$. If furthermore $R$ is complete and reduced, then $\ulrichcat_I(R)$ has finite type. 

\end{thm}{}

\begin{proof}
As $\cond$ is a regular ideal, $\overline R$ is $R$-finite. As $\cond\cong \cond^n$ for all $n$, $B(\cond)=\overline R$. On the other hand $B(I)= B(I^s)=B(\cond)$, proving  the first claim. If $R$ is complete and reduced, then $\overline R$ is a product of DVRs, so $\ulrichcat_I(R)= \cm(\overline R)$ has finite type.   
\end{proof}{}

\begin{prop}\label{endI}
Let $R$ be as in \Cref{I-UlrichThm}. Assume that $I$ is a regular ideal. Let $S=\End_R(I)$ (which is a birational extension of $R$). If $M$ is $I$-Ulrich, then $\Hom_R(M,I)\cong \Hom_R(M,S)$.
\end{prop}

\begin{proof}
 We have an exact sequence $\ses{L}{I\otimes M}{IM}$ where $L$ has finite length. Take $\Hom_R(-,I)$ we get an isomorphism $\Hom_R(IM, I)\cong \Hom_R(I\otimes M, I)$. The first is isomorphic to $\Hom_R(M,I)$ as $IM \cong M$, and the second is isomorphic to $\Hom_R(M,\Hom_R(I,I)) = \Hom_R(M,S)$ by Hom-tensor adjointness.
\end{proof}

\begin{cor}\label{corM*}
Let $R$ be as in \Cref{I-UlrichThm} and further assume that $R$ has a canonical ideal $\omega_R$. The following are equivalent: 
\begin{enumerate}
 
\item $M\in \ulrichcat_{\omega_R}(R)$
\item $\Hom_R(M,R)\cong \Hom_R(M,\omega_R)$. 

\end{enumerate}
\end{cor}
\begin{proof}
$(1)$ implies $(2)$ by \Cref{endI}. Conversely, note that $R\cong \End_R(\omega_R)$. Hence, using Hom-Tensor adjointness, statement $(2)$ is the same as $\Hom_R(\omega_RM,\omega_R)\cong \Hom_R(M,\omega_R)$. Hence dualizing with respect to $\omega_R$ and using \Cref{I-UlrichThm} finishes the proof.
\end{proof}

\begin{cor}\label{omegaulrich}
Let $R$ be as in \Cref{I-UlrichThm}. Assume that $R$ has a canonical ideal $\omega_R $ and $M\in \ulrichcat_{\omega_R}(R)$. Then $M$ is reflexive. 
\end{cor}

\begin{proof}
 \Cref{corM*} implies that $M^*\cong M^{\vee}$, where $M^*=\Hom_R(M,R)$ and $M^{\vee}=\Hom_R(M,\omega_R)$. By \Cref{ulrichunderhom}, $M^*$ is still in $\ulrichcat_{\omega_R}(R)$, so we have $M^{**}\cong M^{*\vee}\cong M^{\vee\vee}\cong M$, as desired. 
\end{proof}

\begin{cor}
Let $R$ be as in \Cref{I-UlrichThm}. Suppose that $R$ has a canonical ideal $\omega_R$. Then for large enough $n$, the ideal $I=\omega_R^n$ is reflexive and satisfies $I^* \cong I^{\vee}$.
\end{cor}

 We end this section by looking into the question when the maximal ideal $\m$ is $I$-Ulrich. This will be applied when we discuss almost Gorenstein rings in the last section. 
 
 \begin{prop}\label{propalmostGor}
 Let $(R,\m,k)$ be as in \Cref{I-UlrichThm}. Suppose there is an exact sequence $\ses{R}{I}{k^{{\oplus}s}}$. Then $\m$ is $I$-Ulrich.
 \end{prop}

\begin{proof}
We can assume that $R$ is not regular, for if $R$ is regular the conclusion follows easily. The assumption is equivalent to $I\m\subset (a)$ for some $a\in I$. We need to show that $I\m=a\m$. As $a\m \subset I\m \subset (a)$ and $\ell((a)/a\m)=1$, it is enough to show that $I\m$ is not equal to $(a)$. Suppose that $I\m=(a)$. Then $E=I\m: I\m = (a):(a) = R$. As $\m:\m \subset E$, it follows that $\m:\m=R$. But $\m:\m \cong \m^*$, and as $\m$ is reflexive, it follows that $\m\cong R$, which is impossible if $R$ is not regular.         
\end{proof}

\section{Applications: Strongly reflexive ring extensions}\label{sre}

Throughout this section, we assume that $(R,\m,k)$ is a one-dimensional Cohen-Macaulay local ring. Suppose that $R$ has a canonical ideal $\omega_R$. In this section we are interested in the following question. Let $S$ be a finite extension of $R$. When is any CM $S$-module $R$-reflexive? Of course if $S=R$, this is equivalent to $R$ being Gorenstein, or $\omega_R\cong R$. It turns out that there is a pleasant generalization to any finite extension $S$ that is in $\cm(R)$: $\cm(S)\subset \rf(R)$ if and only if $\omega_RS\cong S$, in other words $S$ is $\omega_R$-Ulrich. 

We start with a useful lemma that will be used repeatedly in the proof of our main theorem. 
\begin{lem}\label{5.1helplemma}
Let $(R,\m,k)$ be a one-dimensional Cohen-Macaulay local ring. Let $S$ be a module finite $R$-algebra such that $S\in \cm(R)$. Let $M\in \cm(S)$.
\begin{enumerate}
\item The map $$F: M\mapsto \Hom_R(M,R)$$ is an $S$-linear, contravariant functor from $\cm(S)$ to $\cm(S)$. 

\item If $M\in \rf(R)$, then $M\cong F(F(M))$ in $\cm(S)$. 

\end{enumerate}
\end{lem}

\begin{proof}
Note that $\Hom_R(M,R)$ is naturally an $S$-module via the action $$(s\cdot f)(m):=f(sm),\hspace{2em} s\in S,m\in M,f\in \Hom_R(M,R)$$ and that this extends the action of $R$. Thus, the conclusion
of part $(1)$ follows. For part $(2)$, notice that the canonical $R$-linear map $M\to M^{**}$ is $S$-linear with respect to the action above. 
\end{proof}

\begin{thm}\label{totallyreflexive}
Let $(R,\m,k)$ be a one-dimensional Cohen-Macaulay local ring. Assume that $R$ has a canonical ideal $\omega_R$. Let $S$ be a module finite $R$-algebra such that $S\in \cm(R)$. The following are equivalent:
\begin{enumerate}
 \item $\cm(S)\subset \rf(R)$. 
 \item $\omega_S \in \rf(R)$.
 \item $\omega_S \cong S^*$. 
 \item $\omega_RS\cong S$.
  \item $S$ is $\omega_R$-Ulrich.
 \item $S$ is reflexive and $\tr(S)\subset b(\omega_R)$.
\end{enumerate} 

\end{thm}

\begin{proof}
Clearly (1) implies (2). For the converse, assume that $\omega_S\in \rf(R)$. Take $M \in \cm(S)$. Then $M^{\vee} = \Hom_R(M, \omega_R) \in \cm(S)$. Take a free $S$-cover of $M^{\vee}$ and apply $^{\vee}$, we obtain an exact sequence  $\ses{M}{(S^n)^{\vee}}{N}$ in $\cm(R)$. Thus $M \in \rf(R)$ by  \Cref{reflem}.

That (3) implies (2) follows by \Cref{reflem}. 
For the converse, assume that $\omega_S \in \rf(R)$. Take a free $S$-cover of $\omega_S^*$ and apply $^*$. From \Cref{5.1helplemma} we get an exact sequence in $\cm(S)$: $\ses{\omega_S}{(S^n)^*}{N}$. (Note that $N$ is in $\cm(S)$ as it is a submodule of a torsionfree $R$-module and also has an $S$-module structure.)  This has to split in $\cm(S)$ (since $\Ext^1_S(N,\omega_S)=0$), so that $\omega_S$ is a direct summand of $(S^n)^*$ in $\cm(S)$. Since $S\in \rf(R)$ by (2) implies (1), $\omega_S^*$ is a direct summand of $S^n$ in $\cm(S)$ using \Cref{5.1helplemma}$(2)$. Thus $\omega_S^*$ is $S$-projective. Since $S$ is semi-local and $\omega_S^*$ is locally free of constant rank, it is $S$-free of rank one \cite[\href{https://stacks.math.columbia.edu/tag/02M9}{Tag 02M9}]{stacks-project}. Now applying \Cref{5.1helplemma}$(1)$ one gets that $\omega_S^{**}$ is isomorphic to $S^*$ as $S$-modules, so as $R$-modules as well and hence $(3)$ follows. 

The equivalence of (3), (4) and (5) follows from \Cref{I-UlrichThm} and \Cref{corM*}. The equivalence of (5) and (6) follows from \Cref{refIUl2} and \Cref{omegaulrich}
\end{proof}

\begin{dfn}\label{dfnstrong}
We shall call an extension of $R$ satisfying the equivalent conditions of  \Cref{totallyreflexive} a {\it strongly reflexive} extension. 
\end{dfn}

\begin{rem}
The notions of reflexive extensions and totally reflexive extensions, over not necessarily commutative rings, have been defined and studied by X. Chen in \cite[Definition 2.3, 3.3]{chen2013}. They are related to but not the same as ours. For instance, a reflexive extension in Chen's notion would require $S\in \rf(R)$ and $\Hom_R(S,R)\cong S$, and would imply $\rf(S)\subset \rf(R)$.  
\end{rem}

Strongly reflexive birational extensions satisfy even more interesting characterizations.

\begin{thm}\label{refthm}Let $(R,\m,k)$ be a one-dimensional Cohen-Macaulay local ring. Let $S\in \br(R)$. Assume $R$ has a canonical ideal $\omega_R$ admitting a principal reduction. Let $a$ be an arbitrary principal reduction of $\omega_R$ and set $K:=a^{-1}\omega_R$ (see \Cref{rm1}). 
The following are equivalent:
\begin{enumerate}
 \item $\cm(S)\subset \rf(R)$. 
 \item $\omega_S \in \rf(R)$.
 \item $\omega_S \cong S^*$. 
 \item $S$ is $\omega_R$-Ulrich.
 \item $S=KS$.
 \item $K\subseteq S$.
 \item $S\in \rf(R)$ and $\cond_R(S)= \tr(S) \subseteq R:K=(a):\omega_R$. 
 \item $S\in \rf(R)$ and $\cond_R(S)= \tr(S) \subseteq \core(\omega_R):_R \omega_R$ (assuming the residue field is infinite). 
 
\end{enumerate} 

\end{thm}

\begin{proof}
We already know that (1) through (4) are equivalent from \Cref{totallyreflexive}. The equivalence of (4) and (5) follows from \Cref{I-ulrich} and that of (4) and (6) from \Cref{B(I)} and \Cref{I-ulrichextension}.
Since $a$ is an arbitrary principal reduction of $\omega_R$ we see that (7) holds if and only if (8) holds, as the core is the intersection of all principal reductions if the residue field is infinite.

The equivalence of (4), (7), (8) follows from \Cref{refIUl2} and \Cref{omegaulrich}.
%The equivalence of $(3), (4), (5), (6), (7)$ follows from \Cref{frlem}, \Cref{1lem} and \Cref{rm1}.
\end{proof}

\begin{rem}\label{rm1}
The condition that $R$ has a canonical ideal with principal reduction is  
satisfied for instance when $\hat R$ is generically Gorenstein with infinite residue field, see \cite{goto2013almost}.
\end{rem}

\begin{cor}\label{refintclosure}
Let $R$ be as in \Cref{totallyreflexive}. Assume that $R$ has a canonical ideal $\omega_R$. Let $Q(R)\hookrightarrow A$ be an extension of the total quotient ring of $R$. Assume that the integral closure of $R$ in $A$, say $\bar{R}^A$, is a finite $R$-module. Then $\bar{R}^A\in \rf(R)$.
\end{cor}
\begin{proof}
From \Cref{condref}, $\bar{R}\in \ulrichcat_{\omega_R}(R)$. Since $\bar{R}^A\in \cm(\bar{R})$, by \Cref{totallyreflexive} $\bar{R}^A\in \rf(R)$.
\end{proof}
\begin{cor}\label{corstronglyreflexive}
Let $R\rightarrow S$ be a finite extension of rings such that $R$ is a generically Gorenstein $(S_2)$ ring of arbitrary dimension and $S$ is $(S_1)$. If the extension $R\rightarrow S$ is strongly reflexive in codimension one, then any finite $(S_2)$ $S$-module $M$ is $R$-reflexive.
\end{cor}
\begin{proof}
Since $R$ satisfies $(S_2)$ and $M$ is a $(S_2)$ $R$-module, $M$ is $R$-reflexive if and only if this is true in codimension one. Since $R$ is generically Gorenstein and $S$ is $(S_1)$, we may apply \Cref{totallyreflexive} to see $M$ is $R$-reflexive in codimension one.
\end{proof}
\begin{cor}\label{intclosurehigh}
Let $R$ be a generically Gorenstein $(S_2)$ ring of arbitrary dimension. Let $Q(R)\hookrightarrow A$ be an extension of the total quotient ring of $R$. Assume that the integral closure of $R$ in $A$, say $\bar{R}^A$, is a finite $R$-module. Then $\bar{R}^A\in \rf(R)$.
\end{cor}
\begin{proof}
Since $R\rightarrow \bar{R}$ is strongly reflexive in codimension one, by \Cref{corstronglyreflexive} $\bar{R}^A\in \rf(R)$.
\end{proof}

\section{Applications: Finiteness results}\label{firstapps}

Throughout this section, we assume that $(R,\m,k)$ is a one-dimensional Cohen-Macaulay local ring. Here we study when certain subsets of interesting ideals and modules are ``finite".  We say that a subset $\mathcal S$ of $\mod(R)$ is {\it of  finite type} if any element of $\mathcal S$ is isomorphic to a direct sum of modules from a finite set in $\mod(R)$. Note that since we sometimes consider sets that are not subcategories which are closed under isomorphism, this notion is a bit broader than the usual notion of ``finite representation type". Representation finiteness of  subcategories of $\cm(R)$ have been studied heavily, and many beautiful connections to the singularities of $R$ have been discovered over the years. Our study suggests that the same promise could hold for reflexive modules. 

Consider the following classes of ideals of $R$:
\begin{align*}
\mathscr{I}(R)&:=\{I~|\text{$I$ is an integrally closed regular ideal}\}\\
\mathscr{I}_{\cond}(R)&:=\{I~|\text{$I$ is an integrally closed regular ideal and } \cond\subseteq I\}\\
\rf_1(R) &:=\{I~|~\text{$I$ is a reflexive regular ideal}\}.\\
\tracecat(R) &:=\{I~|~ I~\text{is a regular trace ideal}\}
\end{align*}

We shall look at the finiteness of these classes of ideals and the interaction between them. Note that from \Cref{integrallyclosedreflexive}, we have that $\mathscr{I}\subseteq \rf_1(R)$ and that $\mathscr{I}_{\cond}\subseteq \reftracecat(R):= \rf_1(R)\cap \tracecat(R)$. 

\subsection{Finiteness of $\tracecat(R)$}
We begin by answering the following question raised by E. Faber in \cite[Question 3.7]{faber2019trace}.

\begin{ques}\label{questionfaber}
Let $R$ be a one-dimensional complete local or graded ring. Are the following equivalent?
\begin{enumerate}
    \item $\cm(R)$ is of finite type.
    \item There are only finitely many possibilities for $\tr(M)$, where $M \in \cm(R)$.
\end{enumerate} 
\end{ques}

The answer to this question is negative. Consider the following example.
\begin{ex}

Let $\ds R=k[[t^e,\dots,t^{2e-1}]]$ where $\overline{R}=k[[t]]$, $k$ infinite and $e\geq 4$. Then the set of trace ideals is finite but $\cm(R)$ is infinite.

\begin{proof}
Here $\ds \mathfrak{c}=\m$. By \Cref{tracecatrem}, there are exactly two trace ideals, $R$ and $\m$. Since $\overline{R}$ is an $\m$-Ulrich $R$-module, $\mu_R(\overline{R})=e(R)$. However, since $e(R)=e\geq 4$, $\cm(R)$ is infinite by \cite[Theorem 4.2]{leuschke2012cohen}.
\end{proof}
\end{ex}

Note here that finitely many trace ideals in a ring can raise some natural classification questions. Of course, a single trace ideal characterizes a DVR. The following proposition provides a strong motivation to classify such rings. 
\begin{prop}\label{TR3}
Let $(R,\m,k)$ be a complete local one-dimensional domain containing an infinite field $k$, so that $\overline{R}=k[[t]]$. Let $e(R)=e$ and let $v$ be the valuation defining $\overline{R}$. 
\begin{enumerate}
    \item  $\#\tracecat(R)=2$ if and only if $\ds R=k[[t^e,\dots,t^{2e-1}]]$.
    \item The following are equivalent
    \begin{enumerate}
    \item $\#\tracecat(R)=3$ 
    \item $\ds R=k[[\alpha t^e,t^c,t^{c+1},\dots, \widehat{t^{2e}},\dots,t^{c+e-1}]]$ where $\alpha$ is a unit of $k[[t]]$ and $e+2\leq c\leq 2e$.
    \item $\ell(R/\cond)=2$
    \end{enumerate}
\end{enumerate}
\end{prop}
\begin{proof}
  Note that every integrally closed ideal in $R$ is of the form ${I}_f:=\{r\in R~|~v(r)\geq f\}$ where $f\in \mathbb{N}$. Then $\m={I}_e$ and let $\cond={I}_c$ where $c$ is chosen maximally, that is $\cond\neq I_{c+1}$.

    For (1), first assume $\#\tracecat(R)=2$. Here $\m$ and $R$ are the only trace ideals and so $\cond =\m$. Hence, $e=c$ and choose $t^e+\sum_i \beta_it^i\in \m$, $\beta_i\in k$, so that it is part of a minimal generating set for $\m$. Since $\cond=\m$, we have that $t^{e+j}+t^j\sum_i\beta_it^i\in \m$ for all $j\geq 1$. Since $R$ is complete, we have that $t^{e+j}\in \m$ for all $j\geq 0$ and thus $R=k[[t^e,\dots,t^{2e-1}]]$.  The other direction is clear using \Cref{tracecatrem}.
    
    For (2), first assume $\#\tracecat(R)=3$. By \Cref{integrallyclosedreflexive}, we get that there are no integrally closed ideals strictly between $\cond$ and $\m$. In other words there does not exist $r\in R$ such that $v(r)=f$ for all $e<f<c$. Since $R$ is complete, $\cond=(t^c,t^{c+1},\dots,t^{c+e-1})$. Thus we can choose a principal reduction $x=t^e+\sum\limits_{i=1}^{c-e-1} k_it^{e+i}\in R$ of $\m$ where $k_i\in k$.  Consider the ideal $I:=(x)+ \cond$. We claim that $I=\m$. Take any element $r\in \m$. If $v(r)>e$, then $r\in \cond\subseteq I$. If $\ds v(r)=e$, to show $r\in I$, after multiplication by a suitable element of $k$ we may assume that $r=t^e+\sum\limits_{i=1}^{c-e-1} b_it^{e+i}$ where $b_i\in k$. If $r-x\neq 0$, then necessarily $e<v(r-x)<c$, which is impossible. Therefore $r=x$ and $I=\m$. Finally we have $2e\geq c$ since there does not exist any element in $R$ with valuation strictly between $e$ and $c$. Since $\cond\neq \m$, we have that $c\geq e+2$.
    
To show (b) implies (c), assume $R$ has the specified form. Then since $\cond=(t^c,t^{c+1},\dots,t^{c+e-1})$, we have that $\m/\cond$ is a cyclic $R$-module. Moreover since $c\leq 2e$, $\m^2\subseteq \cond$ and $\m/\cond$ is a $k$-vector space. Thus $\ell(\m/\cond)=1$, that is $\ell(R/\cond)=2$. 
    
    (c) implies (a) is clear from \Cref{tracecatrem}.
    \end{proof}

\vspace{1em}
\subsection{Finiteness of $\mathscr{I}$ and $\mathscr{I}_{\cond}$}{~}

	\begin{prop}\label{cond_finite}
	Let $R$ be a one-dimensional Cohen-Macaulay local ring. Suppose $\overline{R}$ is a finite $R$-module. Then $\ds \mathscr{I}_{\cond}$ is a finite set. Moreover if $R$ is a complete local domain, $\mathscr{I}$ is of finite type.
\end{prop}

\begin{proof}

Let $\ds \text{MaxSpec}(\overline{R})=\{\mathfrak{n}_1,...,\mathfrak{n}_s\}$.
Since $\cond$ is a regular ideal of $\overline{R}$, choose an irredundant primary decomposition in $\overline{R}$ , $\ds \cond = \cap_{i=1}^s\mathfrak{n}_i^{(r_i)}$ where $I^{(n)}$ denotes the $n^{th}$ symbolic power of an ideal $I$. Since any $J\in \mathscr{I}_{\cond}$ is the contraction to $R$ of an ideal of $\overline{R}$ containing $\cond$, $\ds J= \cap_{i=1}^s(\mathfrak{n}_i^{(s_i)}\cap R)$ where $1\leq s_i\leq r_i$ for each $i$. Thus $\mathscr{I}_{\cond}$ is a finite set.

 Now assume $R$ is a complete local domain, so that $\overline{R}$ is a DVR; thus the elements of $\mathscr{I}$ are totally ordered by inclusion. From the first part of this proposition, it suffices to consider $I\in \mathscr{I}$, $I\subseteq \cond$. Now $I=I\overline{R}\cap R$, but $I\overline{R}\subseteq \cond\overline{R}=\cond$. Therefore $I$ is also an ideal of $\overline{R}$ and there exists $0\neq a\in R$ such that $a\overline{R}=I$. Therefore $I\cong \overline{R}$ as $R$-modules and $\mathscr{I}$ has finite type.
\end{proof}

\subsection{Finiteness of $\rf_1(R)$}

We first note that $\rf(R)$ (in fact $\rf_1(R)$) is of infinite type if $\overline{R}$ is not finitely generated over $R$.

\begin{lem}\label{B(I)finite}
Let $R$ be a one-dimensional Cohen-Macaulay local ring. Let $I$ be a regular ideal. Then $B(I)$ is a finite $R$-module.
\end{lem}
\begin{proof}
From \Cref{ulrichex}, $I^n$ is $I$-Ulrich for sufficiently large $n$. Thus $I^{n+1}\simeq I^n$ for such $n$ by \Cref{I-UlrichThm}. Therefore $\End_R(I^n)$ stabilizes and $B(I)$ is a finite $R$-module.
\end{proof}

\begin{lem}
Let $R$ be a one-dimensional Cohen-Macaulay local ring. Assume $\rf_1(R)$ has finite type and that $R$ admits a canonical ideal $\omega_R$. Then $\overline{R}$ is a finite $R$-module. In particular, $R$ is reduced.
\end{lem}
\begin{proof}
It suffices to show that $\overline{R}$ is a finite $R$-module. Suppose on the contrary that it is not. By \Cref{B(I)finite}, $\overline{R}$ is not a finite $B(\omega_R)$-module. Thus, we can find an infinite chain of rings $S_i$ inside $\overline{R}$, $B(\omega_R)\subsetneq S_1\subsetneq \dots\subsetneq S_i\subsetneq \dots$ such that each $S_i$ is a finite $R$-module. From \Cref{I-ulrichextension}, the $S_i$ are $\omega_R$-Ulrich and hence by \Cref{totallyreflexive}, they are $R$-reflexive. Consider $S_i\subsetneq S_j$, and let if possible $S_i\simeq S_j$ as $R$-modules. Then they are isomorphic as $S_i$-modules as well. By \Cref{faber}, $S_i=\tr_{S_i}(S_j)\subseteq \cond_{S_i}(S_j)$. So $S_i=S_j$, a contradiction. Therefore the $S_i$'s are indecomposable and mutually non-isomorphic and hence, $\rf_1(R)$ is not of finite type.
\end{proof}

We prove next that $\rf_1(R)$ is of finite type when the conductor has small colength.

 Before stating \Cref{small}, we summarize the cases that we will always reduce to in the proof.

\begin{lem}[Reduction Lemma]\label{reductionlemma}
Let $(R,\m,k)$ be a one-dimensional Cohen-Macaulay local ring. Assume $\overline{R}$ is finite over $R$ and let $\cond$ be the conductor ideal. Further assume that $k$ is infinite. For any ideal $I$, consider the following conditions:
\begin{enumerate}
    \item $\ds \mathfrak{c}\subset I$,
    \item $\ds \mathfrak{c}\subsetneq x:_R \overline{I} \subsetneq x:_R I\subsetneq \m$ where $x$ is a principal reduction of $I$.
    \end{enumerate}
    Let $\rf_1'(R):=\{I\in \rf_1(R)~|~ \text{$I$ satisfies (1) and (2) above}\}$. Then $\rf_1(R)$ is of finite type if and only if $\rf_1'(R)$ is of finite type.
   
\end{lem}

\begin{proof}
If $\rf_1(R)$ is of finite type then certainly $\rf_1'(R)$ is of finite type. Conversely assume that $\rf_1'(R)$ is of finite type. Let $I\in \rf_1(R)$ and $I$ not principal. From \Cref{thMai} we may assume that $\ds \mathfrak{c}\subseteq I$. By \Cref{condref}, we have 
$$\mathfrak{c}\subseteq x:_R \overline{I}\subseteq x:_R {I}\subseteq \m$$
If $\ds \mathfrak{c}=x:_R\overline{I}$, then by \Cref{convention_remark_2}, $\ds (\overline{I})^{*}\cong \mathfrak{c}$ and hence $\cond \cong \bar{I}$. But both are trace ideals, and hence $\ds \mathfrak{c}=I=\overline{I}$. Similarly if $x:_R I=\m$, by \Cref{convention_remark_2} we have $I\cong \m^*$. Finally, if $\ds x:_R\overline{I}=x:_R I$, we get $I^*= (\overline{I})^*$ by \Cref{convention_remark_2}. Thus $I=\overline{I}$ and hence $I\in \mathscr{I}_{\cond}$. Combining the above observations and finally using \Cref{cond_finite}, we have that $\rf_1(R)$ is of finite type.
\end{proof}

\begin{thm}\label{small} Let $(R,\m,k)$ be a one-dimensional Cohen-Macaulay local ring. Let $\overline{R}$ be finite over $R$ and let $\cond$ be the conductor ideal. Further assume that $k$ is infinite.  Consider the following.
\begin{enumerate}
\item $\ds \ell(R/\mathfrak{c})\leq 3$

\item $\ds \ell(R/\mathfrak{c})=4$ and $R$ has minimal multiplicity.
\end{enumerate}
Then in all the above cases, $\ds \rf_1(R)$ is of finite type.
\end{thm}

\begin{proof}

$(1)$ follows immediately from \Cref{reductionlemma}.

Suppose $(2)$ holds. Note that by \Cref{reductionlemma}, we can assume that $I\in \rf_1'(R)$ and $\ell(R/x:_R I)=2$ where $x$ is a principal reduction of $I$.

Since $J=x:_R I$ is reflexive and $R$ has minimal multiplicity, by \Cref{colength2minmult} we get that $\ds J$ is integrally closed. Thus, $\ds I\cong J^*$ where $J\in \mathscr{I}_{\cond}$ and  the proof is now complete by \Cref{cond_finite}.
\end{proof}
\begin{cor}\label{TR3cor1}
Let $R$ be a complete one-dimensional local domain containing an infinite field such that $\#\tracecat(R)=3$. Then $\rf_1(R)$ is of finite type. 
%Moreover, $\rf(R)$ is of finite type if and only if $\cm(\End(\m))$ is of finite type. 
\end{cor}
\begin{proof}
This follows from \Cref{TR3} and \Cref{small}.
\end{proof}
\begin{rem}
\Cref{small} is true if we only assume that $|\Min(\hat{R})|\leq |k|$. To see this, first note that since $R$ is one dimensional and CM, $\bar{R}$ is a finite $R$-module if and only if $R$ is analytically unramified ( see for example \cite[Theorem 4.6]{leuschke2012cohen}). In this case, $\bar{\hat{R}}=\bar{R}\otimes_R \hat{R}$, so $\cond\hat{R}\subseteq \cond_{\hat{R}}$. Therefore $l(\hat{R}/\cond_{\hat{R}})\leq l(\hat{R}/\cond\hat{R})=l(R/\cond)$. Since $\hat{R}$ is reduced, the number of maximal ideals in its integral closure is equal to its number of minimal primes. By \cite[Corollary 3.3]{fouli2018generators}, every ideal of $\hat{R}$ admits a principal reduction. Then from the argument in \Cref{small}, $\rf_1(\hat{R})$ has finite type. By \cite[Proposition 2.5.8]{ega1964ements}, $\rf_1(R)$ has finite type.
\end{rem}
\begin{rem}\label{remGoto}
Since $\cond\in \mathscr{I}_{\cond}$, by \Cref{thMai} and \cite[Proposition 2.9]{gototowards} if $R/\cond$ is Gorenstein, then $\rf_1(R)$ is of finite type. Moreover in this case, by \cite[Theorem 3.7]{corso1998integral}, we have that $\mu(\cond)=\mu(\m)$, so that $R$ necessarily has minimal multiplicity here. In particular finiteness of $\rf_1(R)$ in the cases $\ell(R/\cond)\leq 2$ follows from this as well and in these cases, $R$ necessarily has minimal multiplicity.
\end{rem}

\subsection{Finiteness of $\rf(R)$}

In this subsection, we give a criterion for finiteness of $\rf(R)$ and derive that over seminormal singularities, the category of reflexive modules is of finite type.  

\begin{prop}\label{propfinRef}
Let $(R,\m)$ be a one-dimensional Cohen-Macaulay local ring. Let $S=\End_R(\m)$. If $\cm(S)$ is of finite type, then $\rf(R)$ is of finite type. 
\end{prop}

\begin{proof}
Let $M$ be an indecomposable, non-free reflexive module over $R$. Then $\tr(M)\subset \m =\cond_R(S)$, so $M\in \cm(S)$ by \Cref{faber}. Finally, note that if two $S$-modules are isomorphic, then they are also isomorphic as $R$-modules. 
\end{proof}

\begin{cor}\label{correffull}
Assume that $(R,\m)$ is complete, reduced, one-dimensional and the conductor $\cond$ of $R$ is equal to $\m$. Then $\rf(R)$ is of finite type. 
\end{cor}

\begin{proof}
By assumption $S=\End_R(\m) = \bar R$, which is a product of DVRs and therefore $\cm(S)$ is of finite type. Thus, \Cref{propfinRef} applies.
\end{proof}

As a consequence of the above, we can study finiteness of $\rf(R)$ for `seminormal' rings. R. Swan \cite{swan1980seminormality} defined a seminormal ring as a reduced ring $R$ such that whenever $b,c\in R$ satisfy $b^3=c^2$, there is an $a\in R$ with $a^2=b,a^3=c$. For a detailed exposition on various results related to seminormality (including generalizations to the above definition), we refer the reader to \cite{vitulli2011weak}. Seminormality has also been studied in the context of studying $F$-singularities in characteristic $p>0$. For instance, $F$-injective rings constitute a class of examples for seminormal rings \cite[Corollary 3.6]{datta2019permanence}.

\begin{cor}\label{seminormalcor}
Suppose that  $(R,\m)$ is a seminormal complete reduced local ring of dimension one. Then $\rf(R)$ is of finite  type. 
\end{cor}
\begin{proof}
By \cite[Proposition 2.10(1)]{vitulli2011weak} (with $A=R, B=\overline{R}$), we get that $\cond=\m$, so \Cref{correffull} applies.  
\end{proof}

\section{Further applications, examples and questions}\label{furtherapps}

We begin this section by discussing notions of being `close to Gorenstein' as promised in \Cref{closetorem}. Throughout this section, $(R,\m,k)$ denotes a one-dimensional Cohen Macaulay local ring.

\begin{dfn}

$R$ is called \textbf{almost Gorenstein} if $\ds a:_R \omega_R \supseteq \m$ for some principal reduction $a$ of $\omega_R$. $R$ is called \textbf{nearly Gorenstein} if $\tr(\omega_R)\supseteq \m$. 
\end{dfn}
These classes of rings have attracted a lot of attention lately, the reader can refer to \cite{herzog2019trace}, \cite{barucci1997one}, \cite{goto2013almost}, \cite{goto2015almost}, \cite{dao2020trace} amongst other sources.

In our language:
\begin{prop} Assume that $(R,\m)$ is a one-dimensional Cohen-Macaulay local ring which has a canonical ideal $\omega_R$ with some principal reduction $a$.
$R$ is almost Gorenstein if and only if $\m$ is $\omega_R$-Ulrich.
\end{prop}
\begin{proof}
This follows from \Cref{propalmostGor}.
\end{proof}

It is clear from \Cref{traceformula} that in this situation we get $\ds \tr(\omega_R)\supseteq \m$. This provides a proof for a well-known fact that almost Gorenstein rings are nearly Gorenstein. 

One would expect that for rings close to Gorenstein, it would be easier to find reflexive modules. We now give supporting evidence for that statement.

\begin{prop}\label{almostgorprop}
Let $(R,\m)$ be almost Gorenstein and let $I$ be a regular ideal with $S=\End_R(I)$. 

(1) If $S$ is reflexive and strictly larger than $R$, then $\cm(S)\subseteq \rf(R)$. Thus $IM\in \rf(R)$ for any $M\in \cm(R)$. In particular all powers of $I$ are reflexive.

(2) If $I$ is a proper trace ideal, then $IM$ is reflexive for any $M\in \cm(R)$. In particular $I$ and all of its powers are reflexive. 
\end{prop}

\begin{proof}
As $S$ is reflexive and $\cond_R(S)\subset \m\subseteq R: \frac{\omega_R}{a}$ by hypothesis, we are done by \Cref{refthm}. Since $IM\in \cm(S)$, the proof of (1) is complete. For part (2), just note that $S=I^*$ is reflexive and not equal to $R$. 
\end{proof}

In particular, if $R$ is almost Gorenstein, $\m^n$ is reflexive for all $n$. 

\begin{rem} Suppose $R$ is almost Gorenstein but not Gorenstein, and take $I=\omega_R$. $I$ is not reflexive. Note that here $\End_R(I)$ is reflexive but does not contain $R$ properly. Thus the conditions on $I$ in \Cref{almostgorprop} are needed. 

\end{rem}

The following example shows that in general, $\m^2$ can fail to be reflexive.

\begin{ex}
Let $\ds R=k[[t^5,t^6,t^{14}]]$. Then $\m^2  =(t^{10},t^{11},t^{12},t^{15},t^{16},\dots )$. Thus $(\m^2)^* = (1, t^4,t^5,t^6,\dots)$ and  $t^{14}\in (\m^2)^{**}$. But $t^{14}\notin \m^2$.
\end{ex}

With a bit more work one can even find an example where none of $\m^n$, $n\geq 2$ is reflexive. 
\begin{ex}
Let $R=k[[t^6,t^8,t^{11},t^{13},t^{15}]]$. Then $\m^2=(t^{12},t^{14},t^{16}, \dots)$. Thus, $(\m^2)^*=t^{-12}(t^{11},t^{12},t^{13},t^{14},t^{15},t^{16})=t^{-12}\cond$. Thus, $t^{13}\in (\m^2)^{**}$ but $t^{13}\not \in \m^2$, so $\m^2$ is not reflexive. However, note that $t^{6}$ is a minimal reduction of $\m$ and $t^6\m^2=\m^3$. Thus, none of the higher powers of $\m$ can be reflexive. 
\end{ex}

Next, we classify when $\rf(R)$ is of finite type for almost Gorenstein rings. 
\begin{prop}\label{almostgorequiv}
Suppose that $(R,\m)$ is almost Gorenstein. Let $S=\End_R(\m)$. Then $\rf(R)$ is of finite type if and only if $\cm(S)$ is of finite type.

\end{prop}

\begin{proof}
The `if' direction is \Cref{propfinRef}. The other direction follows from \Cref{almostgorprop}$(1)$. 
\end{proof}

\begin{rem}
 Let $S=\End_R(\m)$. Using the notations in \cite{kobayashi2017syzygies}, we thus get $\cm(S)\subset \rf(R) = \syz\cm(R)$, so $\cm(S)= \syz\cm'(R)$ by \Cref{faber}. It follows that  $ \syz\cm'(R)$ has finite type if and only if $\cm(S)$ has finite type. This recovers results by T. Kobayashi \cite[Corollary 1.3]{kobayashi2017syzygies}.
\end{rem}
\begin{prop}\label{Kobayashiobs}
Let $(R,m,k)$ be a one-dimensional, reduced, complete local ring containing $\mathbb{Q}$ and further assume that $k$ is algebraically closed. Consider the following statements. \begin{itemize}
\item[(a)] $\cm(R)$ is  of finite type.
\item[(b)] $\rf(R)$ is  of finite type.
\item[(c)] $\rf_1(R)$ is  of finite type.
\item[(d)] $\reftracecat(R)$ is finite. 
\end{itemize}
If $R$ is Gorenstein, then all the four statements are equivalent. If $R$ is an almost Gorenstein domain, then $(b),(c)$ and $(d)$ are equivalent. 
\end{prop}

\begin{proof} Assume first that $R$ is Gorenstein. Clearly, $(a)\implies(b) \implies (c)\implies (d)$. Now suppose $\cm(R)$ is not of finite type. Then by \cite[Theorem 4.13(ii)(a)]{leuschke2012cohen} and \Cref{goto2020}, we get $\reftracecat(R)$ is not of finite type. This completes the first part of the proof. 
%The proof of (9.5.2) of Yoshino's book shows that there exist infinitely many finite birational extensions of $R$ which are non-isomorphic as $R$-modules.

Next assume that $R$ is an almost Gorenstein domain. We only need to show $(d)\implies (b)$. Assume that $\rf(R)$ is of infinite type, and hence $\cm(S)$ is of infinite type by \Cref{almostgorequiv}, where $S=\End_R(\m)$. Thus, there are infinitely many non-isomorphic finite reflexive birational extensions of $R$ by \cite[Theorem 4.13(ii)(a)]{leuschke2012cohen} and by \Cref{almostgorprop}$(1)$. The proof is now complete using \Cref{goto2020}.
\end{proof}

Next, we classify birational reflexive extensions of $R$ that are Gorenstein. Our result was inspired by  and extends \cite[Theorem 5.1]{goto2013almost}.

\begin{thm}\label{Gorclassificationreflexive}
Suppose that $R$ is a one-dimensional Cohen-Macaulay local ring with a canonical ideal $\omega_R$. Let $S\in \rf(R)$ be a birational extension of $R$. Let $I=\cond_R(S)$. The following are equivalent:
\begin{enumerate}
    \item $S$ is Gorenstein. 
    \item $I$ is $I$-Ulrich and $\omega_R$-Ulrich. That is $I\cong I^2\cong I\omega_R$.
\end{enumerate}
\end{thm}

\begin{proof}
Note that $I$ is a trace ideal by \Cref{goto2020}. 
Suppose (1) holds. Then $S=\End_R(I)=\End_S(I)$ so $I\cong S$ by \cite[Lemma 2.9]{kobayashi2017syzygies}. Thus $I=aS$ for some $a\in I$, necessarily a regular element, and since $S=S^2$, we have $aI=I^2$. By \Cref{xI}, we have $I\cong I^*$. However, since $S$ is Gorenstein, we have $S\cong S^{\vee}$, so $S^*\cong S^{\vee}$. By \Cref{corM*}, $S$, and hence $I$ is $\omega_R$-Ulrich. 

Assume (2). We can assume that $R$ has infinite residue field and thus $I$ has a reduction $a$. Thus $I^2=aI$, and the same argument in the preceding paragraph shows that $I\cong I^*\cong  S$. So $S$ is $\omega_R$-Ulrich, which (by \Cref{corM*}) implies $S^{\vee} \cong S^* \cong S$, thus $S$ is Gorenstein. 
\end{proof}

\begin{cor}\cite[Theorem 5.1]{goto2013almost}
Suppose that $(R,\m)$ is a one-dimensional Cohen-Macaulay local ring with a canonical ideal $\omega_R$. Let $S=\End_R(\m)$. The following are equivalent:
\begin{enumerate}
    \item $S$ is Gorenstein.
    \item $R$ has minimal multiplicity and is almost Gorenstein. 
\end{enumerate}
\end{cor}

\begin{proof}
Note that minimal multiplicity and almost Gorenstein are just $\m$ being $\m$-Ulrich and $\omega_R$-Ulrich, respectively.
\end{proof}

\begin{ex} (trace ideals are not always reflexive)
Let $R=k[[t^5,t^6,t^7]]$. Here $\ds \mathfrak{c}=\m^2$. Let $I=(t^5,t^7)$. Then $I\in \tracecat(R)$ but $I^{**}\cong (t^5,t^6,t^7)$ and hence $\ds I\not \in \reftracecat(R)$. 
\end{ex}

\begin{proof}
$\ds \cond=\m^2$ is clear. A straight forward computation gives $J:=(t^5:(t^5,t^7))=(t^5,t^{13},t^{14})$, so that $\tr(I)=II^*=I(t^5)^{-1}J=(t^5,t^7)$. Therefore $I\in \tracecat(R)$. 

However, $((t^5):_R J)=\m$. So $I$ is not reflexive.
\end{proof}
\begin{ex}\label{infinitetrace}Let $R=k[[t^4,t^5,t^6]]$, which is a complete intersection domain of multiplicity $4$. The conductor $\cond$ is $\m^2$, with colength $4$. The set $\reftracecat(R)$ is infinite and classified by \cite[Example 3.4(i)]{goto2020correspondence}. We note a few features that illustrate our results:

\begin{enumerate}
    \item It shows that the category of reflexive (regular) ideals is not of finite type, so our \Cref{small} is sharp, as the conductor has colength $4$ but $R$ does not have minimal multiplicity.
    \item From \Cref{cond_finite}, we know that the class of integrally closed ideals of $R$ is of finite type. Note that every integrally closed ideal in $R$ is of the form ${I}_f:=\{r\in R~|~v(r)\geq f\}$ where $v$ is the valuation defining $k[[t]]$ and $f\in \mathbb{N}$ \cite[Corollary 1.3]{d1997integrally}. It is then clear from the proof of \Cref{cond_finite} that a set of representatives for this class is given by $\{(t^n)_{n\geq 4}=\m, (t^n)_{n\geq 5}, (t^n)_{n\geq 6}, (t^n)_{n\geq 8} = \cond \}$.
    \item The rest is the infinite family $\{I_a= (t^4-at^5,t^6),a\in k\}$. We have $t^{10} = (t^4-at^5)t^6+at^5t^6$ which shows that $t^5\in \overline{I_a}$ and so $\overline{I_a}=\m$. So none of the $I_a$ are integrally closed. However let $S=\End_R(\m) = k[[t^4,t^5,t^6,t^7]]$. Then $\ell(S/I_a)=3$ and $\ell(R/I_a)=2$, so $I_aS\cap R=I_a$ by \Cref{contract}. In other words all trace ideals in $R$ are contracted from $S$. 
\end{enumerate}
\end{ex}

\begin{ex}Let $\ds R=k[[t^4,t^6,t^7,t^9]]$. Then clearly $\ds \cond=(t^6,t^7,t^8,t^9)$ and therefore $\tracecat(R)=\{\cond,\m,R\}$. Thus we are in the situation of \Cref{TR3cor1}. 

Next note that $\ds \End_R(\m)=k[[t^2,t^3]]$. By \cite[Theorem 4.18]{leuschke2012cohen}, $\cm(\End(\m))$ is of finite type. Hence $\rf(R)$ is of finite type by \Cref{propfinRef}. 
\end{ex}

\begin{ex}[\textbf{Reflexivity is not preserved under going modulo a non-zero divisor in general}]

 Let $M\in \ulrichcat(R)\cap \rf(R)$. Let $l$ be a principal reduction of $\m$. Then $M/lM$ is a finite dimensional $k$-vector space. Since, $R/l$ is Artinian, $k$ is reflexive if and only if $R$ is Gorenstein (recall that $\Hom_R(k,R)$ is the non-zero socle if $R$ is Artinian). So, if $R$ is not Gorenstein,  then $M/lM\not\in \rf(R/l)$.
 
 \end{ex}

We end the paper with a number of open questions. The following finiteness questions are very natural:
\begin{ques}\label{finitetracetype}
Let $R$ be a complete Cohen-Macaulay local ring of dimension one.
\begin{enumerate}
    \item Can we classify when $R$ has finitely many trace ideals?
    \item Can we classify when $\rf_1(R)$ is of finite type?
    \item Can we classify when $\rf(R)$ is of finite type?
\end{enumerate} 
\end{ques}

Since the trace of a reflexive ideal is a natural place to look for reflexive trace ideals, we have the following question. 

\begin{ques}\label{traceofreflexiveideal}
Let $R$ be a Cohen-Macaulay local ring of dimension one. If an ideal $I\subseteq R$ is reflexive, is $\tr(I)$ reflexive?
\end{ques}
As evidenced by \Cref{small} and \Cref{infinitetrace}, the first obstruction seems to be in colength two.
\begin{ques}\label{questionfaber1}
Let $R$ be a Cohen-Macaulay  local ring of dimension one. When is an ideal of colength two a trace ideal? When is it reflexive?
\end{ques}

We conclude with the following question about rings of positive characteristic. 

\begin{ques}
Let $R$ be a Cohen-Macaulay  local ring of dimension one. 
Suppose $R$ has characteristic $p>0$. When does $R^{1/q}$ belong to $\rf(R)$ for $q=p^i$ large enough? 
\end{ques}

The answer is always positive for one-dimensional complete local domains with algebraically closed residue field where for large $q$,  $R^{1/q}$ becomes a module over $\overline{R}$ and thus $R$-reflexive. 

\section*{Acknowledgments}
 This project is a part of the Kansas Extension Seminar (\url{http://people.ku.edu/~hdao/CAS.html}). We thank the Department of Mathematics at the University of Kansas for partial support. The first author was partly supported by Simons Collaboration Grant FND0077558. We thank Craig Huneke, Daniel Katz, Takumi Murayama and Ryo Takahashi for helpful comments on the topics of the paper. We are also grateful to Toshinori Kobayashi for pointing out \Cref{Kobayashiobs}, Souvik Dey for helpful discussions regarding \Cref{reflem} and the anonymous referees for their suggestions in improving the exposition of this article. 
%\bibliography{references1}

\begin{thebibliography}{GOT{\etalchar{+}}16}
\expandafter\ifx\csname urlstyle\endcsname\relax
  \providecommand{\doi}[1]{doi:\discretionary{}{}{}#1}\else
  \providecommand{\doi}{doi:\discretionary{}{}{}\begingroup
  \urlstyle{rm}\Url}\fi

\bibitem[AH96]{aberbach1996theorem}
\textsc{Ian Aberbach and Craig Huneke}.
\newblock \emph{A theorem of {B}rian{\c{c}}on-{S}koda type for regular local
  rings containing a field}.
\newblock \emph{Proceedings of the American Mathematical Society},
  124(3):707--713, 1996.

\bibitem[Bas60]{bass84}
\textsc{Hyman Bass}.
\newblock \emph{Finitistic dimension and a homological generalization of
  semi-primary rings}.
\newblock \emph{Trans. Amer. Math. Soc.}, 95:466--488, 1960.
\newblock ISSN 0002-9947.
\newblock \doi{10.2307/1993568}.

\bibitem[BF97]{barucci1997one}
\textsc{Valentina Barucci and Ralf Fr{\"o}berg}.
\newblock \emph{One-dimensional almost gorenstein rings}.
\newblock \emph{Journal of Algebra}, 188(2):418--442, 1997.

\bibitem[BH98]{bruns_herzog_1998}
\textsc{Winfried Bruns and H.~Jürgen Herzog}.
\newblock \emph{Cohen-Macaulay Rings}.
\newblock Cambridge Studies in Advanced Mathematics. Cambridge University
  Press, 2 edition, 1998.
\newblock \doi{10.1017/CBO9780511608681}.

\bibitem[BHU87]{brennan1987maximally}
\textsc{Joseph~P Brennan, J{\"u}rgen Herzog and Bernd Ulrich}.
\newblock \emph{Maximally generated cohen-macaulay modules}.
\newblock \emph{Mathematica Scandinavica}, 61:181--203, 1987.

\bibitem[Bou65]{bourbaki1965diviseurs}
\textsc{N~Bourbaki}.
\newblock \emph{Diviseurs, in ``{A}lg{\`e}bre {C}ommutative"}.
\newblock \emph{Chap. VII. Hermann, Paris}, 1965.

\bibitem[BP95]{barucci1995biggest}
\textsc{Valentina Barucci and Kerstin Pettersson}.
\newblock \emph{On the biggest maximally generated ideal as the conductor in
  the blowing up ring}.
\newblock \emph{manuscripta mathematica}, 88(1):457--466, 1995.

\bibitem[Che13]{chen2013}
\textsc{Xiao-Wu Chen}.
\newblock \emph{Totally reflexive extensions and modules}.
\newblock \emph{J. Algebra}, 379:322--332, 2013.
\newblock ISSN 0021-8693.
\newblock \doi{10.1016/j.jalgebra.2013.01.014}.

\bibitem[CHKV05]{corso2005integral}
\textsc{Alberto Corso, Craig Huneke, Daniel Katz and Wolmer~V Vasconcelos}.
\newblock \emph{Integral closure of ideals and annihilators of homology}.
\newblock \emph{Commutative algebra}, 244:33--48, 2005.

\bibitem[CHV98]{corso1998integral}
\textsc{Alberto Corso, Craig Huneke and Wolmer~V Vasconcelos}.
\newblock \emph{On the integral closure of ideals}.
\newblock \emph{manuscripta mathematica}, 95(1):331--347, 1998.

\bibitem[DD97]{d1997integrally}
\textsc{MARCO D'ANNA and Donatella Delfino}.
\newblock \emph{Integrally closed ideals and type sequences in one-dimensional
  local rings}.
\newblock \emph{The Rocky Mountain journal of mathematics}, 27(4):1065--1073,
  1997.

\bibitem[D{\etalchar{+}}58]{dieudonne1958remarks}
\textsc{Jean Dieudonn{\'e} et~al.}
\newblock \emph{Remarks on quasi-frobenius rings}.
\newblock \emph{Illinois Journal of Mathematics}, 2(3):346--354, 1958.



\bibitem[DKT20]{dao2020trace}
\textsc{Hailong Dao, Toshinori Kobayashi and Ryo Takahashi}.
\newblock \emph{Trace of canonical modules, annihilator of ext, and classes of
  rings close to being gorenstein}.
\newblock \emph{arXiv preprint arXiv:2005.02263}, 2020.

\bibitem[DM19]{datta2019permanence}
\textsc{Rankeya Datta and Takumi Murayama}.
\newblock \emph{Permanence properties of $ f $-injectivity}.
\newblock \emph{arXiv preprint arXiv:1906.11399}, 2019.

\bibitem[Fab19]{faber2019trace}
\textsc{Eleonore Faber}.
\newblock \emph{Trace ideals, {N}ormalization {C}hains, and {E}ndomorphism
  {R}ings}.
\newblock \emph{arXiv preprint arXiv:1901.04766}, 2019.

\bibitem[FO18]{fouli2018generators}
\textsc{Louiza Fouli and Bruce Olberding}.
\newblock \emph{Generators of reductions of ideals in a local noetherian ring
  with finite residue field}.
\newblock \emph{Proceedings of the American Mathematical Society},
  146(12):5051--5063, 2018.

\bibitem[GHK03]{gototowards}
\textsc{Shiro Goto, Futoshi Hayasaka and Satoe Kasuga}.
\newblock \emph{Towards a theory of {G}orenstein {$\mathfrak{ m}$}-primary
  integrally closed ideals}.
\newblock In \emph{Commutative algebra, singularities and computer algebra
  ({S}inaia, 2002)}, volume 115 of \emph{NATO Sci. Ser. II Math. Phys. Chem.},
  pages 159--177. Kluwer Acad. Publ., Dordrecht, 2003.

\bibitem[GIK20]{goto2020correspondence}
\textsc{Shiro Goto, Ryotaro Isobe and Shinya Kumashiro}.
\newblock \emph{Correspondence between {T}race {I}deals and birational
  extensions with application to the analysis of the {G}orenstein property of
  rings}.
\newblock \emph{J. Pure Appl. Algebra}, 224(2):747--767, 2020.
\newblock ISSN 0022-4049.
\newblock \doi{10.1016/j.jpaa.2019.06.008}.

\bibitem[GMP13]{goto2013almost}
\textsc{Shiro Goto, Naoyuki Matsuoka and Tran~Thi Phuong}.
\newblock \emph{Almost {G}orenstein rings}.
\newblock \emph{Journal of Algebra}, 379:355--381, 2013.

\bibitem[GOT{\etalchar{+}}14]{goto2014ulrich}
\textsc{Shiro Goto, Kazuho Ozeki, Ryo Takahashi, Kei-Ichi Watanabe and Ken-Ichi
  Yoshida}.
\newblock \emph{Ulrich ideals and modules}.
\newblock In \emph{Mathematical Proceedings of the Cambridge Philosophical
  Society}, volume 156, pages 137--166. Cambridge University Press, 2014.

\bibitem[GOT{\etalchar{+}}16]{goto2016ulrich}
\textsc{Shiro Goto, Kazuho Ozeki, Ryo Takahashi, Kei-ichi Watanabe and Ken-ichi
  Yoshida}.
\newblock \emph{Ulrich ideals and modules over two-dimensional rational
  singularities}.
\newblock \emph{Nagoya Mathematical Journal}, 221(1):69--110, 2016.

\bibitem[Gro67]{ega1964ements}
\textsc{A.~Grothendieck}.
\newblock \emph{\'{E}l\'{e}ments de g\'{e}om\'{e}trie alg\'{e}brique. {IV}.
  \'{E}tude locale des sch\'{e}mas et des morphismes de sch\'{e}mas {IV}}.
\newblock \emph{Inst. Hautes \'{E}tudes Sci. Publ. Math.}, (32):361, 1967.
\newblock ISSN 0073-8301.

\bibitem[GTT15]{goto2015almost}
\textsc{Shiro Goto, Ryo Takahashi and Naoki Taniguchi}.
\newblock \emph{Almost gorenstein rings--towards a theory of higher dimension}.
\newblock \emph{Journal of Pure and Applied Algebra}, 219(7):2666--2712, 2015.

\bibitem[HHS19]{herzog2019trace}
\textsc{J{\"u}rgen Herzog, Takayuki Hibi and Dumitru~I Stamate}.
\newblock \emph{The trace of the canonical module}.
\newblock \emph{Israel Journal of Mathematics}, 233(1):133--165, 2019.

\bibitem[HK71]{herzog238kanonische}
\textsc{J{\"u}rgen Herzog and E~Kunz}.
\newblock \emph{Der kanonische modul eines {C}ohen-{M}acaulay-{R}ings (sem.,
  regensburg, 1970/71)}.
\newblock \emph{Lecture Notes in Math}, 238, 1971.

\bibitem[HS06]{Huneke-Swanson}
\textsc{Craig Huneke and Irena Swanson}.
\newblock \emph{Integral Closure of Ideals, Rings, and Modules}, volume 336 of
  \emph{London Mathematical Society Lecture Note Series}.
\newblock Cambridge University Press, Cambridge, 2006.
\newblock ISBN 978-0-521-68860-4; 0-521-68860-4.

\bibitem[HUB91]{herzog1991linear}
\textsc{Jurgen Herzog, Bernd Ulrich and J{\"o}rgen Backelin}.
\newblock \emph{Linear maximal cohen-macaulay modules over strict complete
  intersections}.
\newblock \emph{Journal of Pure and Applied Algebra}, 71(2-3):187--202, 1991.

\bibitem[Kob17]{kobayashi2017syzygies}
\textsc{Toshinori Kobayashi}.
\newblock \emph{Syzygies of {C}ohen-{M}acaulay modules over one dimensional
  {C}ohen-{M}acaulay local rings}.
\newblock \emph{arXiv preprint arXiv:1710.02673}, 2017.

\bibitem[KT19a]{kobayashi2019rings}
\textsc{Toshinori Kobayashi and Ryo Takahashi}.
\newblock \emph{Rings whose ideals are isomorphic to trace ideals}.
\newblock \emph{Mathematische Nachrichten}, 292(10):2252--2261, 2019.

\bibitem[KT19b]{kobayashi2019ulrich}
\textsc{Toshinori Kobayashi and Ryo Takahashi}.
\newblock \emph{Ulrich modules over cohen--macaulay local rings with minimal
  multiplicity}.
\newblock \emph{Quarterly Journal of Mathematics}, 70(2):487--507, 2019.

\bibitem[KV18]{kustin2018totally}
\textsc{Andrew~R Kustin and Adela Vraciu}.
\newblock \emph{Totally reflexive modules over rings that are close to
  gorenstein}.
\newblock \emph{Journal of Algebra}, 2018.

\bibitem[Lin17]{haydee_trace_end}
\textsc{Haydee Lindo}.
\newblock \emph{Trace ideals and centers of endomorphism rings of modules over
  commutative rings}.
\newblock \emph{J. Algebra}, 482:102--130, 2017.
\newblock ISSN 0021-8693.
\newblock \doi{10.1016/j.jalgebra.2016.10.026}.

\bibitem[Lip71]{lipman1971stable}
\textsc{Joseph Lipman}.
\newblock \emph{Stable ideals and arf rings}.
\newblock \emph{American Journal of Mathematics}, 93(3):649--685, 1971.

\bibitem[LW12]{leuschke2012cohen}
\textsc{Graham~J Leuschke and Roger Wiegand}.
\newblock \emph{Cohen-Macaulay representations}.
\newblock 181. American Mathematical Soc., 2012.

\bibitem[Mai20]{maitra2020partial}
\textsc{Sarasij Maitra}.
\newblock \emph{Partial trace ideals and {B}erger's {C}onjecture}.
\newblock \emph{arXiv preprint arXiv:2003.11648}, 2020.

\bibitem[Mas98]{masek1998gorenstein}
\textsc{VLADIMIR Masek}.
\newblock \emph{Gorenstein dimension of modules}.
\newblock \emph{arXiv preprint math.AC/9809121}, 1998.

\bibitem[Mor58]{morita1958duality}
\textsc{Kiiti Morita}.
\newblock \emph{Duality for modules and its applications to the theory of rings
  with minimum condition}.
\newblock \emph{Science Reports of the Tokyo Kyoiku Daigaku, Section A},
  6(150):83--142, 1958.

\bibitem[NY17]{nakajima2017ulrich}
\textsc{Yusuke Nakajima and Ken-ichi Yoshida}.
\newblock \emph{Ulrich modules over cyclic quotient surface singularities}.
\newblock \emph{Journal of Algebra}, 482:224--247, 2017.

\bibitem[PU05]{polini2005formula}
\textsc{Claudia Polini and Bernd Ulrich}.
\newblock \emph{A formula for the core of an ideal}.
\newblock \emph{Mathematische Annalen}, 331(3):487--503, 2005.

\bibitem[Ser65]{serre1965algebre}
\textsc{Jean-Pierre Serre}.
\newblock \emph{Alg{\`e}bre locale. multiplicit{\'e}s, volume 11 of cours au
  coll{\`e}ge de france, 1957--1958, r{\'e}dig{\'e} par pierre gabriel. seconde
  {\'e}dition, 1965}.
\newblock \emph{Lecture Notes in Mathematics. Springer-Verlag, Berlin}, 6,
  1965.

\bibitem[Ser97]{serre1997algebre}
\textsc{Jean-Pierre Serre}.
\newblock \emph{Alg{\`e}bre locale, multiplicit{\'e}s: cours au Coll{\`e}ge de
  France, 1957-1958}, volume~11.
\newblock Springer Science \& Business Media, 1997.

\bibitem[{Sta}21]{stacks-project}
\textsc{The {Stacks project authors}}.
\newblock \emph{The stacks project}.
\newblock \url{https://stacks.math.columbia.edu}, 2021.

\bibitem[Swa80]{swan1980seminormality}
\textsc{Richard~G Swan}.
\newblock \emph{On seminormality}.
\newblock \emph{Journal of Algebra}, 67(1):210--229, 1980.

\bibitem[Ulr84]{ulrich1984gorenstein}
\textsc{Bernd Ulrich}.
\newblock \emph{Gorenstein rings and modules with high numbers of generators}.
\newblock \emph{Mathematische Zeitschrift}, 188(1):23--32, 1984.

\bibitem[Vit11]{vitulli2011weak}
\textsc{Marie~A Vitulli}.
\newblock \emph{Weak normality and seminormality}.
\newblock In \emph{Commutative Algebra}, pages 441--480. Springer, 2011.

\end{thebibliography}
%\bibliographystyle{alpha1}

\end{document}